\RequirePackage{fix-cm}
\documentclass[smallextended]{svjour3}       
\smartqed  
\usepackage{graphicx}
\usepackage[text={6.6in,9.6in},centering]{geometry}
\usepackage{mathtools,mathrsfs, amsfonts}
\usepackage{comment}
\usepackage{color, ulem}
\usepackage{wrapfig}
\usepackage{floatrow}
\usepackage{graphicx,epic,tikz,pifont}
\usepackage{caption}
\usepackage{natbib}
\usepackage{appendix}

\usepackage{hyperref}
\hypersetup{
    colorlinks,%
    citecolor=blue,%
    filecolor=black,%
    linkcolor=blue,%
    urlcolor=blue
}

\newcommand{\pp}{\mathbb{P}}
\newcommand{\ee}{\mathbb{E}}

\newcommand{\tsp}{\mathcal{T}'}
\newcommand{\rtsp}{\mathcal{T}}
\newcommand{\pe}{E^{\circ}}

\newcommand{\pf}{A}
\newcommand{\ch}{B}

\newcommand{\var}{\mathbb{V}}
\newcommand{\cov}{Cov}

\newcommand{\epf}{\hfill $\square$}

\newcommand{\rev}[1]{#1}

\newcommand{\FF}{\mbox{${\mathcal F}$}}
\newcommand{\NN}{\mbox{${\mathcal N}$}}

\newcommand{\bu}{{\mathbf u}}
\newcommand{\wconv}{\xrightarrow{~d~} }
\newcommand{\asconv}{\xrightarrow{~a.s.~}}
\newcommand{\rconv}{\xrightarrow{~r~}}
\newcommand{\mconv}{\xrightarrow{~1~}}
\newcommand{\pconv}{\xrightarrow{~p~}}
\newcommand{\udim}{d}  
\newcommand{\mass}{t}
\newcommand{\mds}{\xi}
\newcommand{\trf}{\mathbf{B}} 
\newcommand{\bzero}{\mathbf{0}}
\newcommand{\pev}{s}
\newcommand{\cwt}{\mathbf{w}} 
\newcommand{\egt}{\alpha} 
\newcommand{\aegt}{\beta} 

\usepackage[linewidth=1pt]{mdframed}
\hypersetup{
	colorlinks = false,
	urlcolor=false,
	urlbordercolor={white},
	citebordercolor={white},
	linkbordercolor = {white},	
}

\begin{document}

\title{On asymptotic joint distributions of cherries and pitchforks for random phylogenetic trees
}

\titlerunning{On joint subtree distributions of random trees}        

\author{Kwok Pui Choi        \and
        Gursharn Kaur \and Taoyang Wu
}


\institute{Kwok Pui Choi and  Gursharn Kaur\at          
             Department of Statistics and Applied Probability, and the Department of Mathematics, National University of Singapore, Singapore 117546 \\
              \email{stackp@nus.edu.sg, stagk@nus.edu.sg}           
           \and
          Taoyang Wu \at
              School of Computing Sciences, University of East Anglia, Norwich, NR4 7TJ, U.K. \\
 \email{ taoyang.wu@uea.ac.uk}
}

\date{Received: date / Accepted: date}

\maketitle
\normalem 
\begin{abstract}
Tree shape statistics provide valuable quantitative insights into evolutionary mechanisms underpinning phylogenetic trees, a commonly used graph representation of evolution systems ranging from viruses to species. By developing limit theorems for a version of extended P\'olya urn models in which negative entries are permitted for their replacement matrices, we present strong laws of large numbers and  central limit theorems for asymptotic joint distributions of two subtree counting statistics, the number of cherries and that of pitchforks, for random phylogenetic trees generated by two widely used null tree models: the proportional to distinguishable arrangements (PDA) and the Yule-Harding-Kingman (YHK) models. Our results indicate that the limiting behaviour of these two statistics, when appropriately scaled, are independent of the initial trees used in the tree generating process. 


\end{abstract}

\section{Introduction}

As a common mathematical representation of evolutionary relationships among  biological systems ranging from viruses to species, phylogenetic trees retain important signatures of the underlying evolutionary events and mechanisms which are often not directly observable, such as rates of speciation and expansion~\citep{mooers2007some,heath2008taxon}. 
To utilise these signatures, one popular approach is to compare empirical shape indices computed from trees inferred from real datasets with those predicted by neutral models specifying a tree generating process
~\citep[see, e.g.][]{blum2006random,hagen2015age}. Moreover, topological tree shapes are also informative for understanding several fundamental statistics in population genetics~\citep{ferretti2017decomposing,arbisser2018joint} and important parameters in 
the dynamics of virus evolution and propagation~\citep{colijn2014phylogenetic}. 

Here we will focus on two subtree counting statistics: the number of cherries (e.g.   nodes which have precisely two descendent leaves) and that of pitchforks (e.g.  nodes which have precisely three descendent leaves) in a tree. These statistics are related to monophylogenetic structures in phylogenetic trees~\citep{rosenberg03a} and have been utilised recently to study evolutionary dynamics of pathogens~\citep{colijn2014phylogenetic}. 
Various statistical properties concerning these two statistics have been established for the past decades on the following two fundamental phylogenetic tree sampling models:  the proportional to distinguishable arrangements (PDA) and the Yule-Harding-Kingman (YHK) models~\citep{McKenzie2000,rosenberg06a,chang2010limit, disanto2013exact, WuChoi16,CTW19}.

In this paper we are interested in the limiting behaviour of the joint cherry and pitchfork distributions for the YHK and the PDA models. 
In a seminal paper, \citet{McKenzie2000} showed that cherry distributions converge to a normal distribution, which was later extended to pitchforks and other subtrees by ~\citet{chang2010limit}.  More recently, \citet{Janson2014} studied subtree counts in the random binary search tree model, and their results imply that the cherry and pitchfork distributions converge jointly to a bivariate normal distribution under the YHK model. This is further investigated in~\citet{WuChoi16} and~\citet{CTW19}, where numerical results indicate that convergence to bivariate normal distributions holds under both the YHK model and the PDA model. Our main results here provide a unifying  approach to establishing the convergence of the joint distributions  to bivariate normal distributions for both models, as well as a strong law stating that the joint counting statistics  converge almost surely (a.s.) to a constant vector.  
 
  Our approach is based on a general model in probability theory known as the  P\'olya urn scheme, which has been developed during the past few decades including applications in studying various growth phenomena with an underlying random tree structure (see, e.g.~\citet{Hosam2009} and the references therein). For instance, the results in ~\cite{McKenzie2000} are based on a version of the urn model in which the off-diagonal elements in the replacement matrix are all positive. However, such technical constraints pose a central challenge for studying pitchfork distributions as negative entries in the resulting replacement matrix are not confined only to the diagonal (see Sections~\ref{sec:results_yhk} and~\ref{sec:results_pda}). To overcome this limitation, here we study  a family of extended P\'olya urn models under certain technical assumptions in which negative entries are allowed for their replacement matrices (see Section~\ref{sec:urn}). Inspired by the martingale approach used in~\cite{BaiHu2005}, we present a self-contained proof for the limit theorems for this extended urn model, with the dual aims of completeness and accessibility. 
Note that our approach is  different from one popular framework in which discrete urn models are embedded into a continuous Markov chain known as the branching processes (see, e.g.~\citet{Janson2004} for some recent developments).


We now summarize the contents of the rest of the paper. In the next section, we collect some definitions concerning phylogenetic trees and the two tree-based Markov processes. Then, in Section~\ref{sec:urn}, we present an introduction to the urn model and a version of the Strong Law of Large Numbers and the Central Limit Theorem that are applicable to our study. Using these two theorems, we present our results for  the YHK process in Section~\ref{sec:results_yhk}, and those for  the PDA process in Section~\ref{sec:results_pda}. These results are extended to unrooted trees in Section~\ref{sec:unrooted}. 
 The proofs of the main results for the urn model are presented in Section~\ref{sec:tech:proofs}, with a technical lemma included in the appendix. We conclude in the last section with a discussion of our results and some open problems.

\section{Preliminaries} \label{sec:preliminaries}

In this section, we present some basic notation and background concerning phylogenetic trees, random tree models, and urn models. From now on $n$ will be a positive integer greater than two unless stated otherwise.

\subsection{Phylogenetic Trees}
 A {\em tree} $T=\left( V(T),E(T) \right)$ is a connected acyclic graph with vertex set $V(T)$ and edge set $E(T)$. A vertex is referred to as a {\em leaf} if it has degree one, and an {\em interior vertex} otherwise.  An edge incident to a leaf is called a {\em pendant edge}, and let $\pe(T)$ be the set of pendant edges in $T$.
  A tree is {\em rooted} if it contains exactly one distinguished degree one node designated as the {\em root}, which is not regarded as a leaf and is usually denoted by $\rho$, and {\em unrooted} otherwise.  Other than those in Section~\ref{sec:unrooted},  all trees considered here are rooted and  {\em binary}, that is, each interior vertex has precisely two children. 

 \begin{figure}[ht]
\begin{center}
{\includegraphics[scale=0.8]{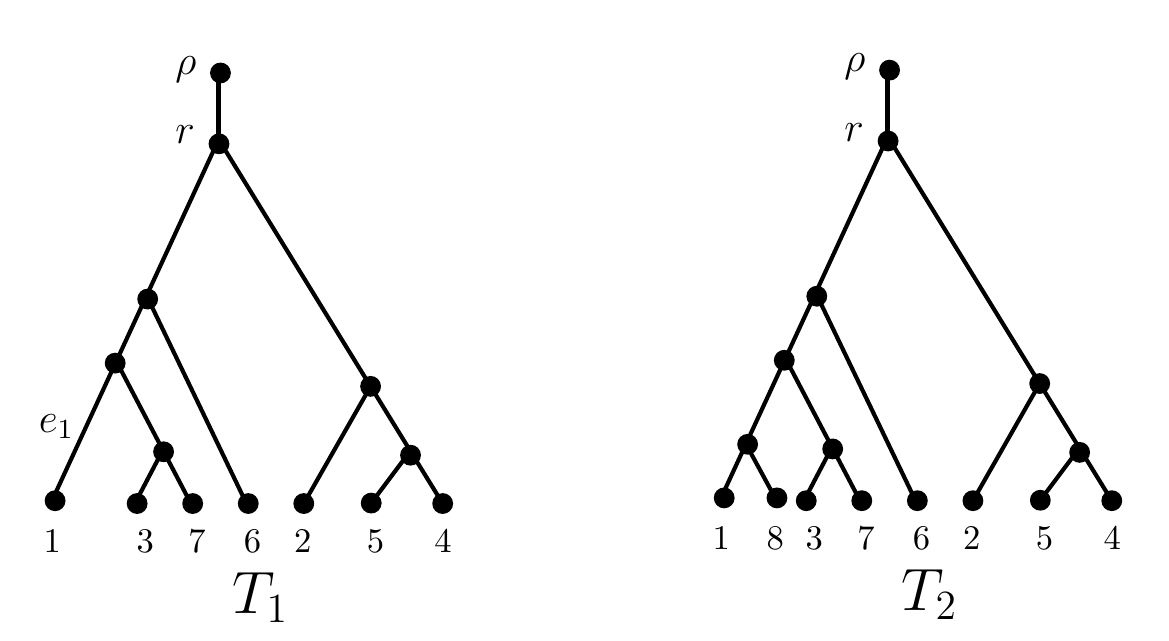}}
\end{center}
\caption{Examples of phylogenetic trees.  $T_1$ is a rooted phylogenetic tree on $\{1,\dots,7\}$; $T_2=T_1[e_{1}]$ is a phylogenetic tree on $X=\{1,\dots,8\}$ obtained from $T_1$ by attaching a new leaf labelled $8$ to the edge $e_{1}$ which is incident with taxon $1$ in $T_1$. 
}
\label{fig:tree_examples}
\end{figure}

 A {\em phylogenetic tree} on a finite set $X$ is a rooted  tree with leaves bijectively labelled by the elements of $X$. The set of binary rooted phylogenetic trees on $\{1,2,\dots,n\}$ is denoted by $\rtsp_n$. 
See Fig.~\ref{fig:tree_examples} for examples of trees in $\rtsp_7$ and  $\rtsp_8$. Given an edge $e$ in a phylogenetic tree $T$ on $X$ and a taxon $x' \not \in X$, let $T[e;x']$ be the phylogenetic tree on $X\cup \{x'\}$ obtained by attaching a new leaf with label $x'$ to the edge $e$. Formally, let $e=(u,v)$ and let $w$ be a vertex not contained in $V(T)$. Then $T[e;x']$ has vertex set $V(T) \cup \{x',w\}$ and edge set $\big(E(T) \setminus \{e\} \big) \cup \{(u,w), (v,w), (w,x')\}$. See Fig.~\ref{fig:tree_examples} for an illustration of this construction, \rev{where tree $T_2=T_1[e_{1};8]$ is obtained from $T_1$ by attaching leaf $8$ to the edge $e_{1}$}. Note that we also use $T[e]$ instead of $T[e;x']$ when the taxon name $x'$ is not essential.

Removing an edge in a phylogenetic tree $T$ results in two connected components; the connected component that does not contain the root of $T$ is referred to as a subtree of $T$. A subtree is called a {\em cherry} if it has two leaves, and a {\em pitchfork} if it has three leaves.  Given a phylogenetic tree $T$, let $\pf(T)$ and $\ch(T)$ be the number of  pitchforks and cherries contained in $T$. For example, in Fig.~\ref{fig:tree_examples} we have $\pf(T_2)=1$ and $\ch(T_2)=3$.

\subsection{The YHK and the PDA Processes}
\label{subsection:model}
Let $\rtsp_n$ be the set of  phylogenetic trees with $n$ leaves.
In this subsection, we introduce  the two tree-based Markov processes investigated in this paper:  the  proportional to distinguishable arrangements (PDA) process  and the  
Yule-Harding-Kingman (YHK) process, which is largely based on~\citet{CTW19} and adapted from  the Markov processes as described in~\citet[Section 3.3.3]{steel2016phylogeny}.

Under the YHK process~\citep{yule25a, harding71a}, starting with a given tree $T_m$  in $\rtsp_m$ with $m\ge 2$, a  random phylogenetic tree $T_n $ in $\rtsp_n$ is generated as follows.
\begin{itemize}
\item[(i)] Select a uniform random permutation $(x_1,\dots,x_n)$ of $\{1,2,\dots,n\}$;
\item[(ii)] label the leaves of the rooted phylogenetic tree $T_m$ randomly using the taxon set $\{x_1,x_2,\cdots,x_m\}$;
\item[(iii)] for $m\le k <n$, uniformly choose a random pendant edge $e$ in $T_k$ and let 
$T_{k+1}=T_k[e;x_{k+1}]$.
\end{itemize}
Here a permutation $(x_1,\dots,x_n)$ of $\{1,2,\dots,n\}$ means a taxon sequence with $x_i \in \{1,2,\dots,n\}$ and $x_i\not =x_j$ for all $i\not =j$. 
The PDA process can be described using a similar scheme; the only difference is that in Step (iii) the edge $e$ is uniformly sampled from the edge set of $T_k$, instead of the pendant edge set. Furthermore, under the PDA process, Step (i) can also be simplified by using a fixed permutation, say $(1,2,\cdots,n)$.
In the literature, the special case $m=2$, for which $T_2$ is the unique tree with two leaves, is also referred to as the YHK model and the PDA model, respectively.


For $n\ge 4$, let $\pf_n$ and $\ch_n$  be the random variables $\pf(T)$ and  $\ch(T)$, respectively, for a random tree $T$ in $\rtsp_n$. The probability distributions of $\pf_n$ (resp. $\ch_n$) will be referred to as pitchfork distributions (resp. cherry distributions). 
In  this paper, we are mainly interested in the limiting  distributional properties of $(\pf_n, \ch_n)$.

\subsection{Modes of Convergence}

Let $X, X_1,X_2,\dots$ be random variables on some probability space $(\Omega,\FF,\pp)$. To study the urn model we will use the following four modes of convergence~(see, e.g.~\citet[Section 7.2]{grimmettprobability} for more details). First,  $X_n$ is said to converge to $X$ {\em almost surely}, denoted as $X_n \asconv X$, if $\{\omega \in \Omega\,:\, X_n(\omega)\to X(\omega)~\mbox{as}~n \to \infty \}$ is an event with probability 1. Next, $X_n$ is said to converge to $X$  {\em in $r$-th mean}, where $r\ge 1$, written $X_n \rconv X$, if $\ee(|X_n^r|)<\infty$ for all $n$ and $\ee(|X_n-X|^r) \to 0$ as $n\to \infty$. 
Furthermore, $X_n$ is said to converge to $X$  {\em in probability}, written $X_n \pconv X$, if $\pp(|X_n-X| > \epsilon) \to 0$ as $n\to \infty$ for all $\epsilon>0$. 
Finally, $X_n$ converges to a random variable $Y$ {\em in distribution}, also termed {\em weak convergence} or {\em convergence in law} and written $X_n \wconv Y$, if $\pp(X_n \le x) \to \pp(Y \le x)$ as $n\to \infty$ for all points $x$ at which the distribution function $\pp(Y\le x)$ is continuous. Note that 
$X_n \pconv X$ implies $X_n \wconv X$, and $X_n \pconv X$ holds if either 
 $X_n \asconv X$ holds or $X_n\rconv X$ holds for some $r\ge 1$. 

\subsection{Miscellaneous}

Let $\bzero=(0,\dots,0)$ be the $\udim$-dimensional zero row vector.
Let $\mathbf{e}=(1,\dots,1)$ be the $\udim$-dimensional row vector whose entries are all one, and for $1 \le j \le d$,  let
$ {\mathbf e}_j$ denote  the $j$-th canonical row vector whose $j$-th entry is $1$ while the other entries are all zero.

Let $\text{diag}(a_1,\dots,a_d)$ denote a diagonal matrix whose diagonal elements are $a_1, \ldots , a_d$. Furthermore, $\bzero^\top\bzero$ is the $\udim\times \udim$ matrix whose entries are all zero.
 Here $Z^\top$ denotes the transpose of $Z$, where $Z$ can be either a vector or a matrix.

\section{Urn Models}
\label{sec:urn}


In this section, we briefly recall the classical P\'olya urn model and some of its generalisations. The P\'olya urn model was first studied by  \cite{Polya1930}, and since then it has been applied in describing evolutionary processes in biology and  computer science. Several such applications in genetics  are discussed in \citet[Chapter 5]{Kotz1977} and in~\citet[Chapters 8 and 9]{Hosam2009}.
In a general setup, consider an urn with balls of $\udim$ different colours containing $C_{0,i}$ many balls of colour $i \in \{1,2,\dots, \udim\}$ at time $0$.  At each time  step,  a ball is drawn uniformly at random and returned with some extra balls, depending on the colour selected.
The reinforcement scheme is often described by a $\udim \times \udim$ matrix $R$:  if the colour of the ball drawn is  $i$, then we return the selected ball along with  adding or removing $R_{ij} $ many balls of colour $j$, for every $j \in \{ 1,2,\cdots, \udim\}$. A negative value of $R_{ij}$ corresponds to removing $|R_{ij}|$ many balls from the urn.  Such a matrix is termed as {\em replacement matrix} in the literature. 
For instance,  the replacement matrix $R$ is the identity matrix  for  the original P\'olya urn model with $d$ colours,  that is,  at each time point, the selected ball is returned with one additional ball of the same colour.

Let $C_n = (C_{n,1}, \dots, C_{n,\udim})$  be the row vector of dimension $\udim$ that represents the ball configuration at time $n$ for an urn model with $\udim$ colours. Then the  sum of $C_{n,i}$, denoted by $\mass_n$, is the number of balls in the urn at time $n$. Recall that a vector is referred to as a {\em stochastic vector} if each entry in the vector is a non-negative real number and the sum of its entries is one.  Denote the stochastic vector associated with $C_n$ by $\widetilde{C}_n$, that is, we have $\widetilde{C}_{n,i}=C_{n,i}/\mass_n$ for $1\le i \le \udim$.

 Let $\FF_{n}$ be the information of the urn's configuration from time $1$ up to $n$, that is, the $\sigma$-algebra generated by $C_0,C_1,\cdots,C_n$.
 Let $R$ denote the  replacement matrix. Then, for every $n\geq 1$,
 \begin{equation}
  \label{eq:main:urn:dynamic}
C_{n} = C_{n-1} +\chi_{n} R, 
 \end{equation}
where $\chi_{n} $ is a  random  row vector of length $d$ such that for $i=1, \ldots, d$,
\[\pp(\chi_{n} = {\mathbf e}_i \vert \FF_{n-1}) =
\widetilde{C}_{n-1,i}.
\]
Since precisely one entry in $\chi_{n}$ is $1$ and all others are $0$, it follows that
 \begin{equation}
 \label{eq:chi:prop}
\ee[\chi_n\vert \FF_{n-1}]=\widetilde{C}_{n-1}
~\quad~\mbox{and}~\quad~
\ee[\chi_n^\top\chi_n\vert \FF_{n-1}]=\text{diag}( \widetilde{C}_{n-1}).
 \end{equation}

\vspace{1em}

We state  the following  assumptions about the replacement matrix $R$:
\begin{enumerate}
\item [(A1)] {\em  Tenable:} It is always possible to draw balls and follow the replacement rule, that is, we never get stuck in following the rules (see, e.g.~\citet[p.46]{Hosam2009}).
\item [(A2)] {\em Small:}
All eigenvalues of $R$ are real; the maximal eigenvalue $\lambda_1=s$ is positive with $\lambda_1>2\lambda$ holds for all other eigenvalues $\lambda$ of $R$.
\item [(A3)]  {\em Strictly Balanced:}
The column vector $\mathbf{e}^\top$ is a right eigenvector of $R$ corresponding to $\lambda_1$ and one of the left eigenvectors corresponding to $\lambda_1$ is a stochastic vector.
\item[(A4)] {\em Diagonalisable:} $R$ is diagonisable over real numbers. That is, there exists \rev{an invertible} matrix $U$ with real entries such that
\begin{equation}
\label{eq:R:diagonal}
U^{-1}RU  = \text{diag}(\lambda_1, \lambda_2,\dots, \lambda_d) =: \Lambda,
\end{equation}
where $\lambda_1\ge \lambda_2 \ge \dots \ge \lambda_d$ are all eigenvalues of $R$.
\end{enumerate}
\vspace{1em}

Note that under assumption (A3) we have $t_n = t_0+ns$, which implies that the urn model is {\em balanced}, as commonly known in the literature. For the matrix $U$ in (A4) and $1\le j \le \udim$,  let $\mathbf{u}_j=U{\mathbf e}^\top_j$ denote the $j$-th column of $U$, and $\mathbf{v}_j={\mathbf e}_jU^{-1}$ the $j$-th row of $U^{-1}$. Then $\mathbf{u}_j$ and $\mathbf{v}_j$ are, respectively, right and left eigenvectors corresponding to $\lambda_j$. Furthermore, since
$\mathbf{v}_i \mathbf{u}_j={\mathbf e}_iU^{-1}   U{\mathbf e}^\top_j
={\mathbf e}_i \mathbf{I}\, {\mathbf e}^\top_j$, where $\mathbf{I} $ is the identity matrix, 
we have
\begin{equation}
\label{eq:u:v:innerproduct}
\mathbf{v}_i \mathbf{u}_j=1~~\mbox{if $i=j$, and }
\mathbf{v}_i \mathbf{u}_j=0~~\mbox{if $i\not =j$.}
\end{equation}
In view of (A3), (A4) and~\eqref{eq:u:v:innerproduct},  for simplicity the following convention will be used throughout this paper:
\begin{equation}
\label{eq:p-evector}
\mathbf{u}_1=\mathbf{e}^\top
~\quad \mbox{and}~~\quad \mathbf{v}_1~\mbox{is a stochastic vector}.
\end{equation}
Furthermore, the eigenvalue $\lambda_1$ will be referred to as the principal eigenvalue; $\mathbf{u}_1$ and $\mathbf{v}_1$ specified in ~\eqref{eq:p-evector} as the principal right and principal left eigenvector, respectively.


The limit of the urn process and the  rate of convergence  to the limiting vector  depends on the spectral properties of matrix $R$.   Theorems \ref{thm1} and \ref{thm2} below give  the Strong Law of Large Numbers and the Central Limit Theorem of the extended P\'olya urn model under our assumptions (A1)--(A4). 
Our proofs,  which are  adapted from that of \cite{BaiHu2005},  will be presented in Section~\ref{sec:tech:proofs}\,.


\begin{theorem} \label{thm1}
Under assumptions  {\em (A1)--(A4)}, we have
\begin{equation}
\label{eq:asconv:urn}
(n\pev)^{-1}  C_n \asconv  \mathbf{v}_1
~\quad~\mbox{and} ~\quad~
(n\pev)^{-1}  C_n \rconv  \mathbf{v}_1
~\quad~\mbox{for $r >0$},
\end{equation}
where $\pev$ is the principal eigenvalue and $\mathbf{v}_1$ is the principal left eigenvector.
\end{theorem}


\medskip

Let  $\mathcal{N}(\mathbf{0}, \Sigma)$ be the multivariate normal distribution with mean vector $\mathbf{0}$ and  covariance matrix $\Sigma$.

\begin{theorem} \label{thm2}
Under assumptions  {\em (A1)--(A4)},  we have
$$
n^{-1/2} ( C_n - n\pev \mathbf{v}_1) \wconv \NN(\mathbf{0}, \Sigma),
$$
where $\pev$ is the principal eigenvalue, $\mathbf{v}_1$ is the principal left eigenvector, and
\begin{equation}
\Sigma =  \sum_{i,j=2}^d \frac{\pev \lambda_i \lambda _j  \bu_i^\top \mbox{\em diag}(\mathbf{v}_1) \bu_j  }{\pev-\lambda_i -\lambda_j}  \mathbf{v}_i^\top \mathbf{v}_j.
\end{equation}
\end{theorem}


\section{Limiting Distributions under the YHK Model}
\label{sec:results_yhk}
A cherry is said to be {\em independent} if it is not contained in any pitchfork, and {\em dependent} otherwise. Similarly, a pendant edge is {\em independent} if it is contained in neither a pitchfork nor a cherry.
In this section, we study the limiting joint distribution of the random variables $A_n$ (i.e., the number of pitchforks) and $B_n$ (i.e., the number of cherries) under the YHK model.


To study the joint distribution of cherries and pitchforks, we extend the urn models used in~\citet{McKenzie2000} (see also \citet[Section 3.4]{steel2016phylogeny}) as follows. Each pendant edge in a phylogenetic tree is designated as one of the following four types:
\begin{itemize}
	\item[(E1):]  a type $1$ edge is a pendant edge in a dependent cherry (i.e, contained in both a cherry and a pitchfork);
	\item[(E2):]  a type $2$ edge is a pendant edge in an independent cherry; 
	\item[(E3):]  a type $3$ edge is a pendant edge contained in a pitchfork but not a cherry; 
	\item[(E4):]  a type  $4$ edge is an independent pendant edge (i.e, contained in neither a pitchfork nor a cherry). 	
\end{itemize}

 \begin{figure}[ht]
	\begin{center}
		{\includegraphics[width=0.9\textwidth]{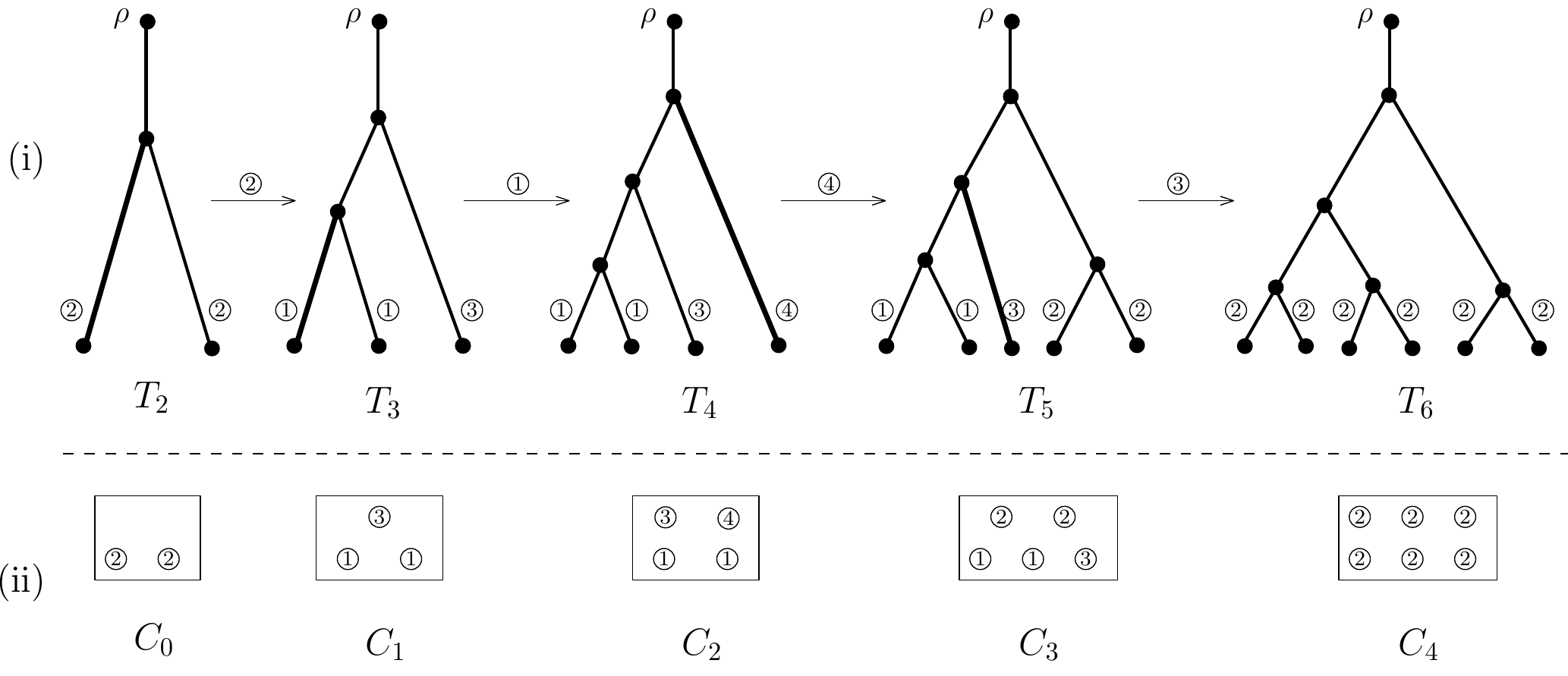}}
	\end{center}
	\caption{A sample path of the YHK model and the associated urn model. (i): A sample path of the YHK model evolving from $T_2$ with two leaves to $T_6$ with six leaves. The labels of the leaves are omitted for simplicity. The type of pendant edges is indicated by the circled numbers next to them. For $2\le i \le 5$, the edge selected in $T_i$ to generate $T_{i+1}$ is highlighted in bold and the associated edge type is indicated in the circled number above the arrows. (ii) The associated urn model with four colours, derived from the types of pendants edges in the trees.  
	 Note that in the vector form we have $C_0=(0,2,0,0), C_1=(2,0,1,0), C_2=(2,0,1,1), C_3=(2,2,1,0)$ and $C_4=(0,6,0,0)$. 
	}
	\label{fig:yhk_dynamics}
\end{figure}


It is straightforward to see that any  pendant edge in a phylogenetic tree with at least two leaves belongs to one and only  one of the above four types.  Furthermore, the numbers of pitchforks and independent cherries in a tree are precisely half of the numbers of type-1 and type-2 edges, respectively. 

As illustrated in Fig.~\ref{fig:yhk_dynamics},  the composition of the types of the pendant edges in $T[e]$, the tree obtained from $T$ by attaching an extra leaf to a pendant edge $e$, is determined by the composition of pendant edge types in $T$ and the type of $e$ as follows.
When  $e$ is  type 1 , then the number of type 4 edges in $T[e]$ increases by one compared with that in $T$ while the number of edges of each of the other three types is the same. This holds because both $T[e]$ and $T$ have the same number of cherries and that of pitchforks (see $T_3$ and $T_4$ in Fig.~\ref{fig:yhk_dynamics}).   When
$e$ is of type 2, then the number of type-2 edges decreases by two while the numbers of type 1 and of type 3 increase by two and one, respectively. This is because in this case one independent cherry is replaced by one pitchfork. When $e$ is type 3, one pitchfork is replaced by two independent cherries, hence  the number of type 2 edges increases by four while the numbers of edges of type 1 and of type-3 decrease by two and one, respectively. Finally, when $e$ is type 4,  one independent pendant edge is replaced by one independent cherry, and hence the number of type 2 edges increases by two and that of type 4 edges decreases by one.  

Using the dynamics described in the last paragraph, we can associate a YHK process starting with a tree $T_m$ with a corresponding urn process $(C_0,R)$ as follows. The urn model contains four colours in which colour $i$ ($1\le i \le 4$) is designated for type $i$ edges. In the initial urn $C_0=(C_{0,1},\cdots,C_{0,4})$, the number $C_{0,i}$ is precisely the number of type $i$ edges in $T_m$. Furthermore,  the replacement matrix $R$ is the following $4 \times 4$ matrix:
\begin{equation} 
\label{Matrix:R}
 R = \begin{bmatrix*}[r]
0 ~& 0 ~&0~&1\\
2~&-2~&1~&0\\
-2~&4~&-1~&0\\
0~&2~&0~&-1
\end{bmatrix*}.
\end{equation}

Given an arbitrary tree $T$, let $\egt(T)=\big(|E_1(T)|, |E_2(T)|,|E_3(T)|,|E_4(T)|\big)$ be the pendant type vector associated with $T$ where $|E_i(T)|$ counts  the number of type $i$ edges in $T$ 
for $1\le i \le 4$.

The following result will enable us to obtain the joint  distribution on pitchforks and cherries for the YHK model. 

\begin{theorem}
	\label{thm:yhk:color}
	Suppose that $T_m$ is an arbitrary phylogenetic tree with $m$ leaves with $m\ge 2$, and that $T_n$ is a tree with $n$ leaves generated by the YHK process starting with $T_m$.  
	 Then we have 
	\begin{equation}
	\frac{\egt(T_n)}{n} \asconv  
	\mathbf{v}_1	
	~\quad~\mbox{and}~\quad~
	\frac{ \egt(T_n) - n\mathbf{v}_1 }{\sqrt{n}} \wconv \mathcal{N}\left (\bzero,\Sigma \right ),
	\end{equation}
	where $\mathbf{v}_1=\big(\frac{2}{6},\frac{2}{6},\frac{1}{6}, \frac{1}{6} \big)$ and
	\begin{equation}
	\label{eq:sigma:yhk:urn}
		\Sigma  = \frac{1}{1260} \begin{bmatrix*}[r]
	\,276 ~&~ -388  ~&~ 138 ~&~ -26\,\\ 
	\,-388 ~ &~ 724 ~&~ -194 ~&~ -142\, \\
	\,138 ~&~ -194  ~&~ 69  ~&~ -13 \,\\
\,	-26 ~&~ -142 ~&~  -13 ~&~ 181 \,\end{bmatrix*}. 
 \end{equation}
\end{theorem}

\begin{proof}
Consider the YHK process $\{T_{n}\}_{n\ge m}$ starting with $T_m$. 
Let $C_k=\egt(T_{k-m})$ for $k\ge m$. Then $C_k = (C_{k,1}, \dots, C_{k,4})$, with $C_{k,i}=|E_i(T_{k-m})|$, is the urn model of $4$ colours derived from the pendant edge decomposition of the YHK process. Therefore, it is a tenable  model starting with  $C_0=\egt(T_m)$ and replacement  matrix $R$ as given in~\eqref{Matrix:R}. 

Note that $R$ is diagonalisable as 
\[U^{-1}R U=\Lambda  \]
holds with
\begin{equation} \label{Decomp-YHK}
U = \begin{bmatrix*}[r]
1 & 1&-1&-1\\
1&0&-1&-3\\
1&-2&2&5\\
1&0&2&3
\end{bmatrix*},  
\qquad 
\Lambda = \begin{bmatrix*}[r]  1&0&0 & 0 \\ 0&0&0&0& \\ 0&0 &-2&0 \\ 0&0&0&-3\end{bmatrix*} 
\qquad \text {and }  \qquad 
U^{-1}=  \frac{1}{6}  \begin{bmatrix*}[r]
2&2& 1&1\\
2&-2&-2&2\\
-4&2&-2&4\\
2&-2&1&-1
\end{bmatrix*}.
\end{equation}
Therefore, $R$ satisfies condition (A4). Next, (A2) holds because  $R$  has eigenvalues  
$$s=\lambda_1=1,~\quad~ \lambda_2=0, ~\quad~\lambda_3=-2, ~\quad~\lambda_4=-3,$$
where $s=\lambda_1=1$ is the principal eigenvalue. Furthermore, put $\mathbf{u}_i=U\mathbf{e}^\top_i$ and $\mathbf{v}_i=\mathbf{e}_iU^{-1}$ for $1\le i \le 4$. Then  (A3) follows by noting that $\mathbf{u}_1=(1,1,1,1)^\top$ is the principal right eigenvector, and $\mathbf{v}_1=\frac{1}{6}\big(2,2,1,1\big)$ is the principal left eigenvector. 

Since (A1)--(A4) are satisfied by the replacement matrix $R$, by Theorem~\ref{thm1} it follows that
$$
	\frac{C_k}{k} \asconv 
\mathbf{v}_1
~~\mbox{with}~~ {k\to \infty}
$$
and hence 
$$
\frac{\egt(T_n)}{n}=\frac{n-m}{n} 	\frac{C_{n-m}}{n-m} \asconv 
\mathbf{v}_1
~~\mbox{with}~~ {n\to \infty}.$$
By Theorem \ref{thm2} we have 
\begin{equation}
\label{eq:pda:urn:pf:slln}
\frac{ C_{n-m} - (n-m)\mathbf{v}_1 }{\sqrt{n-m}}=
 \dfrac{ C_k - k  v_1}{\sqrt{k}} \wconv \mathcal{N}(\bzero,  \Sigma),
 \end{equation}
 where 
\begin{equation} 
\Sigma =  \sum_{i,j=2}^4 \frac{ \lambda_i \lambda _j  \bu_i^\top \mbox{\em diag}(\mathbf{v}_1) \bu_j  }{1-\lambda_i -\lambda_j}  \mathbf{v}_i^\top \mathbf{v}_j.
\end{equation}
Therefore, we have
\begin{align*}
\frac{ \egt(T_n) - n\mathbf{v}_1 }{\sqrt{n}} 
&=\frac{ C_{n-m} - (n-m)\mathbf{v}_1 }{\sqrt{n}}+\frac{m\mathbf{v}_1}{\sqrt{n}} \\
&=\frac{\sqrt{n-m}}{\sqrt{n}}\frac{ C_{n-m} - (n-m)\mathbf{v}_1 }{\sqrt{n-m}}
+\frac{m\mathbf{v}_1}{\sqrt{n}} \\
&\wconv \mathcal{N}\left (\bzero,\Sigma \right ).
\end{align*}
Here the convergence follows from~\eqref{eq:pda:urn:pf:slln} and the fact that
$\frac{\sqrt{n-m}}{\sqrt{n}}$ converges to $1$ and $\frac{m\mathbf{v}_1}{\sqrt{n}}$ convergences to $0$ when $n$ approaches infinity. 
\epf
\end{proof}

By Theorem~\ref{thm:yhk:color}, it is straightforward to obtain the following result on the joint distribution of cherries and pitchforks, which also follows a general result in~\citep[Theorem 1.22]{Janson2014} .

\begin{corollary}\label{cor:yhk}
	Under the YHK model, for the joint distribution $(A_n,B_n)$ of pitchforks and  cherries we have
	\begin{equation}
	\frac{1}{n}(\pf_n , \ch_n) \asconv  \Big(\frac{1}{6}, \frac{1}{3} \Big) 
	\end{equation}
	and
	\begin{equation}
	\frac{ (\pf_n , \ch_n ) - n ( 1/6, 1/3)}{\sqrt{n}} \wconv \mathcal{N}\left (\bzero,\frac{1}{1260 } 
	\begin{bmatrix*}[r]
	69 ~& -28\\
	-28 ~& 56
	\end{bmatrix*} \right ).
	\end{equation}
\end{corollary}	

\begin{proof}
Consider the YHK process $\{T_{n}\}_{n\ge 2}$ starting with a tree $T_2$ with two leaves. Denote the $i$-th entry in $\egt(T_n)$ by $\egt_{n,i}$ for $1\le i \le 4$.  Then the corollary follows from Theorem~\ref{thm:yhk:color} by noting that we have $A_n=\frac{\egt_{n,1}}{2}$ and 
$B_n=\frac{\egt_{n,1}+\egt_{n,2}}{2}$. 
\epf	
\end{proof}	

The above result is consistent with the previously known results on the mean and (co-)variance of the joint distribution of cherries and pitchforks (see, e.g.,~\cite{WuChoi16,CTW19} ), namely, under the YHK model and for $n\ge 7$ we have
$$
\ee(A_n)=\frac{n}{6}, 
\quad
\ee(B_n)=\frac{n}{3}, 
\quad
\var(A_n)=\frac{23n}{420}, 
\quad
\var(B_n)=\frac{2n}{45},~~\mbox{and}~~ \cov(A_n,B_n)=-\frac{n}{45}.
$$


\section{Limiting Distributions under the PDA Model}
 \label{sec:results_pda}

In this section, we study the limiting joint distribution of the random variables $\pf_n$ (i.e., the number of pitchforks) and $\ch_n$ (i.e., the number of essential cherries) under the PDA model. 

To study PDA model, in addition to the four edge types (E1)-(E4) considered in Section~\ref{sec:results_yhk}, which partitions the set of pendant edges, we
need two additional edge types concerning the internal edges. Specifically, 
 \begin{itemize}
	\item[(E5):]  a type $5$ edge is an internal edge adjacent to an independent cherry;
	\item[(E6):]  a type $6$ edge is an internal edge that is not type $5$.   
\end{itemize}

For $1\le i \le 6$, let $E_i(T)$ be the set of edges of type $i$. Then the edge sets $E_1(T),\dots,E_6(T)$ form a partition of the edge set of $T$. That is, each edge in $T$ belongs to one and only one $E_i(T)$. Furthermore, let $\aegt(T)=\big(|E_1(T)|, \dots,|E_6(T)|\big)$ be the type vector associated with $T$, where $|E_i(T)|$ counts  the number of type $i$ edges in $T$.

 \begin{figure}[ht]
	\begin{center}
		{\includegraphics[width=0.9\textwidth]{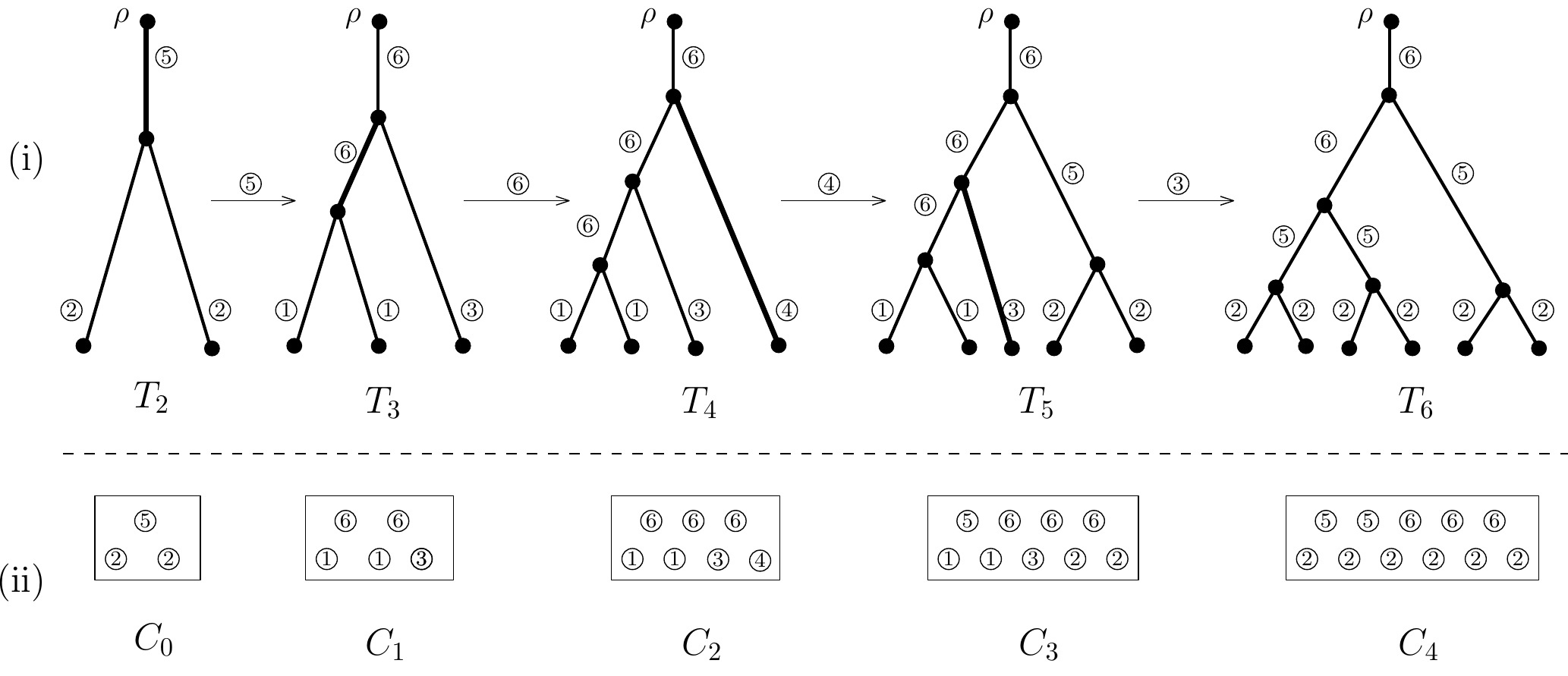}}
	\end{center}
	\caption{A sample path of the PDA model and the associated urn model. (i) A sample path of the PDA model evolving from $T_2$ with two leaves to $T_6$ with six leaves. The labels of the leaves are omitted for simplicity. The edge types are indicated by circled numbers. For $2\le i \le 5$, the edge selected in $T_i$ to generate $T_{i+1}$ is highlighted in bold and the associated edge type is indicated in the circled number above the arrows. (ii) The associated urn model with six colours, derived from the edge types in the trees. Note that in the vector form we have $C_0=(0,2,0,0,1,0), \ldots, C_3=(2, 2, 1, 0,3)$ and $C_4=(0,6,0,0,2,3)$. 		
	}
	\label{fig:pda_dynamics}
\end{figure}

As illustrated in Fig.~\ref{fig:pda_dynamics},  the composition of the edge types  in $T[e]$, which is  obtained from $T$ by attaching an extra leaf to edge $e$, is determined by the composition of edge types in $T$ and the type of $e$. First, if $e$ is a pendant edge, the change of the composition of the pendant edge types  in $T[e]$ is the same as described in Section~\ref{sec:results_yhk}, and the change of the composition of the interior edge types in $T[e]$ is described as follows: 
\begin{itemize}
	\item[(i)] If $e$ is type-1, then $|E_i(T[e])|-|E_i(T)|$ is $0$ if $i=5$, and $1$ if $i=6$;
	\item[(ii)] if $e$ is type-2, then $|E_i(T[e])|-|E_i(T)|$ is $-1$ if $i=5$, and $2$ if $i=6$;
	\item[(iii)] if $e$ is type-3, then $|E_i(T[e])|-|E_i(T)|$ is $2$ if $i=5$, and $-1$ if $i=6$;
	\item[(iv)] if $e$ is type 4, then $|E_i(T[e])|-|E_i(T)|$ is $1$ if $i=5$, and $0$ if $i=6$.
\end{itemize}
Finally, when $e$ is type-5, the change it caused is the same of that of a type-2 edge, and when $e$ is type 6, the change it caused is the same of that of type-1 ege. Therefore, we can associate a PDA process starting with a tree $T_0$ with a corresponding urn process $(C_0,R)$ as follows. The urn model contains six colours in which colour $i$ ($1\le i \le 6$) is designated for type $i$ edges. In the initial urn $C_0=(C_{0,1},\cdots,C_{0,6})$, the number $C_{0,i}$ is precisely the number of type $i$ edges in $T_0$. Furthermore,  the replacement matrix $R$ is the following $6 \times 6$ matrix:
\begin{equation} \label{Matrix:R:PDA}
R = \begin{bmatrix*}[r]
0~&0 ~&0 ~&1 ~&0 ~&1\\
2~&-2~&1 ~&0 ~&-1 ~&2\\
-2~&4~&-1 ~&0 ~&2 ~&-1\\
0~&2~&0   ~&-1 ~&1 ~&0\\
2~&-2~&1 ~&0 ~&-1 ~&2\\
0~&0~&0 ~&1 ~&0 ~&1
\end{bmatrix*}.
\end{equation}

Note that the replacement matrix for the YHK model in~\eqref{Matrix:R} is a submatrix of the replacement matrix in~\eqref{Matrix:R:PDA}; and the last (respectively, second last) row  in~\eqref{Matrix:R:PDA} is the same as its first (respectively, second) row. These two observations are direct consequences of the dynamic described above. The theorem below describes the asymptotic behaviour of  $\aegt(T_n)$, which enables us to deduce the asymptotic properties  of  the joint  distribution of the number of  pitchforks and the number of cherries for the PDA model in Corollary \ref{cor:pda}. 

\begin{theorem}
	\label{thm:pda:color}
	Suppose that $T_m$ is an arbitrary phylogenetic tree with $m$ leaves with $m\ge 2$, and that $T_n$ is a tree with $n$ leaves generated by the PDA process starting with $T_m$.  
	Then we have 
	\begin{equation}
	\frac{\aegt(T_n)}{n} \asconv  
	\mathbf{v}_1	
	~\quad~\mbox{and}~\quad~
	\frac{ \aegt(T_n) - n\mathbf{v}_1 }{\sqrt{n}} \wconv \mathcal{N}\left (\bzero,\Sigma \right ),
	\end{equation}
	as $n \to \infty$, where $\mathbf{v}_1= \frac{1}{16}( 2,2,1,3,1,7)$ and 
	\begin{equation}
	\label{eq:sigma:pda:urn}
 \Sigma  = \frac{1}{64} \begin{bmatrix*}[r]
12  ~& -12  ~&6 ~&-6 ~& -6  ~&6\\
-12  ~& 28 ~&-6 ~&-10 & 14 ~& -14 \\
6  ~& -6  ~&3 &-3 ~&-3  ~&3 \\
-6 ~&  -10 ~&-3 ~& 19 ~&-5  ~&5 \\
-6  ~&  14 ~&-3 ~&-5  ~&7~& -7\\
6  ~&-14 ~& 3 ~& 5 ~& -7  ~&7
\end{bmatrix*}. 
	\end{equation}
\end{theorem}

\begin{proof}
	Consider the PDA process $\{T_{n}\}_{n\ge m}$ starting with $T_m$. 
	Let $C_k=\aegt(T_{k-m})$ for $k\ge m$. Then $C_k = (C_{k,1}, \dots, C_{k,6})$ with $C_{k,i}=|E_i(T_{k-m})|$ is the urn model of $6$ colours derived from the edge partition of the PDA process. Therefore, it is a tenable  model starting with  $C_0=\aegt(T_m)$ and replacement  matrix $R$ as given in~\eqref{Matrix:R:PDA}. 
	
	Note that $R$ is diagonalisable as 
	\[U^{-1}R U=\Lambda  \]
	holds with $\Lambda=\text{diag}(2,0,0,0,-2,-4)$ and
\begin{equation} \label{Decomp-PDA1}
U = \begin{bmatrix*}[r]
1 ~& 2.5 ~&2 ~&1 ~&1 ~&1\\
1~&-2~&1 ~&0 ~&1 ~&5\\
1 ~&-8 ~&  -1    ~&1 ~&  -3  ~& -9\\
1 ~& -1  ~&  1 ~&   1 ~& -3 ~&  -5 \\
1  ~& 3   ~&-1   ~& 1  ~&  1 ~&   5 \\
1  ~& 1 ~&  -1  ~& -1 ~&   1   ~& 1
\end{bmatrix*}
~\quad~\mbox{and}~\quad~ 
U^{-1} = \frac{1}{176}\begin{bmatrix*}[r]
22  ~& 22 ~& 11 ~& 33  ~& 11 ~& 77 \\
4 ~& -20 ~&  -14 ~& 14 ~& 6 ~& 10\\
30 ~& 26 ~& -17 ~& 17 ~& -43  ~& -13 \\
40 ~& -24 ~& 36 ~&-36 ~& 60 ~& -76\\
66 ~& -22 ~& 33 ~& -77 ~&-11 ~& 11\\
-22 ~& 22 ~&-11 ~&11 ~& 11 ~&-11 \end{bmatrix*}.
\end{equation}
	Therefore, $R$ satisfies condition (A4). Next, (A2) holds because  $R$  has eigenvalues (counted with multiplicity) 
	$$s=\lambda_1=2,~\quad~ \lambda_2=0, ~\quad~\lambda_3=0, ~\quad~\lambda_4=0,~\quad~\lambda_5=-2,~\quad~\lambda_6=-4$$
	where $s=\lambda_1=2$ is the principal eigenvalue. Furthermore, put $\mathbf{u}_i=U\mathbf{e}^\top_i$ and $\mathbf{v}_i=\mathbf{e}_iU^{-1}$ for $1\le i \le 6$. Then  (A3) follows by noting that $\mathbf{u}_1=(1,1,1,1,1,1)^\top$ is the principal right eigenvector, and $\mathbf{v}_1= \frac{1}{16}( 2,2,1,3,1,7)$ is the principal left eigenvector. 
	
	Since (A1)--(A4) are satisfied by the replacement matrix $R$, by Theorem~\ref{thm1} it follows that
	$$
	\frac{C_k}{k} \asconv 
	\mathbf{v}_1
	~~\mbox{with}~~ {k\to \infty}
	$$
	and hence 
	$$
	\frac{\aegt(T_n)}{n}=\frac{n-m}{n} 	\frac{C_{n-m}}{n-m} \asconv 
	\mathbf{v}_1
	~~\mbox{with}~~ {n\to \infty}.
	$$
	By Theorem \ref{thm2} we have 
	\begin{equation}
	\label{eq:pda:urn:pf:slln}
	\frac{ C_{n-m} - (n-m)\mathbf{v}_1 }{\sqrt{n-m}}=
	\dfrac{ C_k - k  v_1}{\sqrt{k}} \wconv \mathcal{N}(\bzero,  \Sigma),
	\end{equation}
	where 
	\begin{equation} 
	\Sigma =  \sum_{i,j=2}^6 \frac{ \lambda_i \lambda _j  \bu_i^\top \mbox{\em diag}(\mathbf{v}_1) \bu_j  }{1-\lambda_i -\lambda_j}  \mathbf{v}_i^\top \mathbf{v}_j.
	\end{equation}
	Therefore, we have
	\begin{align*}
	\frac{ \aegt(T_n) - n\mathbf{v}_1 }{\sqrt{n}} 
	&=\frac{ C_{n-m} - (n-m)\mathbf{v}_1 }{\sqrt{n}}+\frac{m\mathbf{v}_1}{\sqrt{n}} \\
	&=\frac{\sqrt{n-m}}{\sqrt{n}}\frac{ C_{n-m} - (n-m)\mathbf{v}_1 }{\sqrt{n-m}}
	+\frac{m\mathbf{v}_1}{\sqrt{n}} \\
	&\wconv \mathcal{N}\left (\bzero,\Sigma \right ).
	\end{align*}
	Here the convergence follows from~\eqref{eq:pda:urn:pf:slln} and the fact that
	$\frac{\sqrt{n-m}}{\sqrt{n}}$ converges to $1$ and $\frac{m\mathbf{v}_1}{\sqrt{n}}$ converges to $0$ when $n$ approaches infinity. 
	\epf
\end{proof}

Similar to Corollary \ref{cor:yhk}, by Theorem~\ref{thm:pda:color} it is straightforward to obtain the following result on the joint distribution of cherries and pitchforks.

\begin{corollary}\label{cor:pda}
	Under the PDA model, for the joint distribution $(A_n,B_n)$ of pitchforks and  cherries we have
	\begin{equation}
	\frac{1}{n}(\pf_n , \ch_n) \asconv  \Big(\frac{1}{8}, \frac{1}{4} \Big) 
	\end{equation}
	and
	\begin{equation}
	\frac{ (\pf_n , \ch_n ) - n ( 1/8, 1/4)}{\sqrt{n}} \wconv \mathcal{N}\left (\bzero,~
\frac{1}{64}  	 \begin{bmatrix*}[r]
3 ~&~ 0 \\
0 ~&~ 4
\end{bmatrix*} 
	\right )
	\end{equation}
	as $ n \to \infty$.
\end{corollary}	

\begin{proof}
	Consider the PDA process $\{T_{n}\}_{n\ge 2}$ starting with a tree $T_2$ with two leaves. Denote the $i$-th entry in $\aegt(T_n)$ by $\aegt_{n,i}$ for $1\le i \le 6$.  Then the corollary follows from Theorem~\ref{thm:yhk:color} by noting that we have $A_n=\frac{\aegt_{n,1}}{2}$ and 
	$B_n=\frac{\aegt_{n,1}+\aegt_{n,2}}{2}$. 
	\epf	
\end{proof}

The above result is consistent with the previously known results on the mean and (co-)variance of the joint distribution of cherries and pitchforks (see, e.g.,~\cite{WuChoi16,CTW19} ), namely, under the PDA model and for $n\ge 7$ we have
\begin{align*}
\ee(A_n)&=\frac{n(n-1)(n-2)}{2(2n-3)(2n-5)},
~~~~~~\quad
\ee(B_n)=\frac{n(n-1)}{2(2n-5)},
~~
&\var(B_n)=\frac{n(n-1)(n-2)(n-3)}{2(2n-3)^2(2n-7)}, \\
\var(A_n)&=\frac{3(4n^3-40n^2+123n-110)}{2(2n-5)(2n-7)(2n-9)} \, \var(B_n),
~~~~~\mbox{and}~
&\cov(A_n,B_n)=\frac{-\var(B_n)}{(2n-7)}.
\end{align*}

\section{Unrooted Trees}
\label{sec:unrooted}
In this section, we extend our results in Sections ~\ref{sec:results_yhk} and ~\ref{sec:results_pda} to the unrooted version of phylogenetic trees. 
Formally,  deleting the root $\rho$ of a rooted phylogenetic tree and suppressing its adjacent interior vertex $r$ results in an unrooted tree (see Fig.~\ref{fig:unrooted}). The set of unrooted phylogenetic trees on $\{1,2,\dots,n\}$ will be denoted by $\tsp_n$. The YHK process on unrooted phylogenetic tree is similar to that on rooted ones stated in Section~\ref{subsection:model}; the only difference is that at step (ii) we shall start with an unrooted phylogenetic tree $T_m$ in $\tsp_m$ for $m\ge 3$. Similar modification suffices for the PDA processes on unrooted phylogenetic trees; see~\cite{CTW19} for more details. Note that the concepts of cherries and pitchforks can be naturally extended to unrooted trees in $\tsp_n$ for $n\ge 6$.  Moreover, let $A'_n$ and $B'_n$ be the random variables counting the number of pitchforks and cherries in a random tree in $\tsp_n$.

To associate urn models with the two processes on unrooted trees, note that for a tree $T$ in $\tsp_n$ with $n\ge 6$, we can decompose the edges in $T$ into the six types similar to those  for rooted trees, and hence define $\alpha(T)$ and $\beta(T)$ correspondingly.  Furthermore, the replacement matrix is the same as the unrooted one, that is, the replacement matrix for the YHK model is given in~\eqref{Matrix:R}  and the one for the  PDA process is given in~\eqref{Matrix:R:PDA}. See two examples in Fig.~\ref{fig:unrooted}.  We  emphasize that the condition $n\ge 6$ is essential here: for instance, there is no appropriate assignment for the edge $e_2$ in the tree $T_5$ in Fig.~\ref{fig:unrooted} in our scheme, neither type 3 nor type 4 satisfying the requirement of a valid urn model. This observation is indeed in line with the treatment of unrooted trees in~\cite{CTW19}. However, there is only one unrooted shape for $n=4$ and one for $n=5$. 
Furthermore, there are only two tree shapes for $\tsp_6$ (as depicted in $T_6^1$ and $T_6^2$ in Fig.~\ref{fig:unrooted}). In particular,  putting $\alpha_6^1=(4,0,2,0)$ and $\alpha_6^2=(0,6,0,0)$, then for each $T$ in $\tsp_6$, we have either $\alpha(T)=\alpha_6^1$ or $\alpha(T)=\alpha_6^2$.


 \begin{figure}[ht]
	\begin{center}
		{\includegraphics[width=0.9\textwidth]{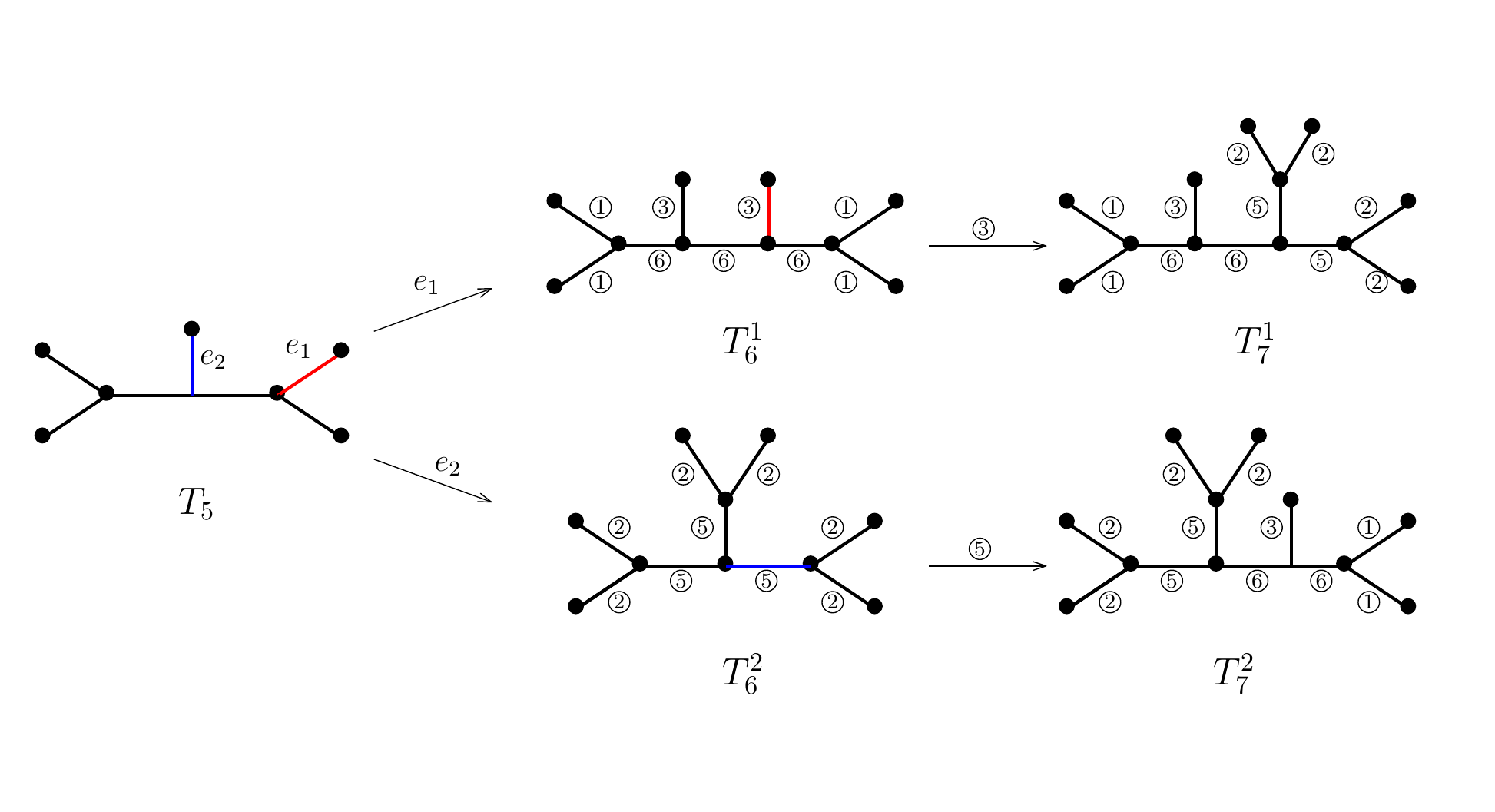}}
	\end{center}
	\caption{Example of sample paths for the PDA process on unrooted trees and the associated urn model. Two sample paths of the PDA process evolving from $T_5$: one ends with $T^1_7$ using the edges in red and the other with $T^2_7$ using the edges in blue. Leave labels are omitted for simplicity. Note that in the vector form we have $\beta(T^1_6)=(4,0,2,0,0,3)$ and $\beta(T^2_6)=(0,6,0,0,3,0)$. 
	}
	\label{fig:unrooted}
\end{figure}

Now we extend Theorem~\ref{thm:yhk:color} and Corollary~\ref{cor:yhk} to the following result concerning the limiting behaviour of the YHK process, 

\begin{theorem}
	\label{thm:yhk:color:unrooted}
	Suppose that $T_m$ is an arbitrary unrooted phylogenetic tree with $m$ leaves with $m\ge 6$, and that $T_n$ is an unrooted tree with $n$ leaves generated by the YHK process starting with $T_m$.  
	Then, as $n \to \infty$, 
	\begin{equation}
	\label{eq:yhk:urn:unrooted}
	\frac{\egt(T_n)}{n} \asconv  
	\mathbf{v}_1	
	~\quad~\mbox{and}~\quad~
	\frac{ \egt(T_n) - n\mathbf{v}_1 }{\sqrt{n}} \wconv \mathcal{N}\left (\bzero,\Sigma \right ),
	\end{equation}
	where $\mathbf{v}_1=\big(\frac{2}{6},\frac{2}{6},\frac{1}{6}, \frac{1}{6} \big)$ and $\Sigma$ is given in Eq.~(\ref{eq:sigma:yhk:urn}). In particular, as $n \to \infty$,  
		\begin{equation}
			\label{eq:yhk:model:unrooted}
	\frac{1}{n}(\pf'_n , \ch'_n) \asconv  \Big(\frac{1}{6}, \frac{1}{3} \Big) 
	~\quad~\mbox{and}~\quad~
	\frac{ (\pf'_n , \ch'_n ) - n ( 1/6, 1/3)}{\sqrt{n}} \wconv \mathcal{N}\left (\bzero,\frac{1}{1260 } 
	\begin{bmatrix*}[r]
	69 ~& -28\\
	-28 ~& 56
	\end{bmatrix*} \right ).
	\end{equation}
\end{theorem}

\begin{proof}
The proof of~\eqref{eq:yhk:urn:unrooted} follows  an argument similar to that for Theorem~\ref{thm:pda:color}.  

To establish~\eqref{eq:yhk:model:unrooted}, consider the YHK process $\{T_{n}\}_{n\ge 2}$ starting with a tree $T_2$ with two leaves. 
For $n\ge 6$, let $\egt_n=\egt(T_n)$ and $\egt_{n,i}$ denote the $i$-th entry in $\egt(T_n)$  for $1\le i \le 4$.  
Consider the vector $\alpha_6^1=(4,0,2,0)$ and $\alpha_6^2=(0,6,0,0)$. For $j=1,2$, let $E_j$ be the event that $\alpha_6=\alpha_6^j$. It follows that $E_1$ and $E_2$ form a partition of the sample  space.  Moreover, we have $\pp(E_1)=4/5$ and $\pp(E_2)=1-\pp(E_1)=1/5$. Consider the random indicator variable $\mathbb{I}_{E_1}$, that is, $\pp(\mathbb{I}_{E_1}=1)=4/5$ and $\pp(\mathbb{I}_{E_1}=0)=1/5$.  Random indicator variable $\mathbb{I}_{E_2}$ is similarly defined. Then we have 
$$
\alpha_n=\alpha^1_n\mathbb{I}_{E_1}+\alpha^2_n \mathbb{I}_{E_2}.
$$
Furthermore, by~\eqref{eq:yhk:urn:unrooted} we have 
$\frac{\alpha^j_n}{n}  \asconv  \mathbf{v}_1$ a.s. on $E_j$, for $j=1,2$, and hence
$$
\frac{\alpha_n}{n}
\asconv \mathbf{v}_1 
(\mathbb{I}_{E_1}+\mathbb{I}_{E_2})=\mathbf{v}_1. 
$$
Together with $A_n'=\frac{\egt_{n,1}}{2}$ and 
$B_n'=\frac{\egt_{n,1}+\egt_{n,2}}{2}$, the almost surely convergence in~\eqref{eq:yhk:model:unrooted} follows. Finally, the convergence in distribution in~\eqref{eq:yhk:model:unrooted}  also follows from a similar argument. 
\epf
\end{proof}	

Finally, combining Theorem~\ref{thm:pda:color}, Corollary~\ref{cor:pda}, and an argument similar to the proof of Theorem~\ref{thm:yhk:color:unrooted} leads  to the following result concerning the limiting behaviour of the unrooted PDA process, whose proof is hence omitted.

\begin{theorem}
	\label{thm:pda:color:unrooted}
	Suppose that $T_m$ is an arbitrary unrooted phylogenetic tree with $m$ leaves with $m\ge 6$, and that $T_n$ is an unrooted tree with $n$ leaves generated by the PDA process starting with $T_m$.  
	Then, as $n \to \infty$, 
	\begin{equation}
	\frac{\aegt(T_n)}{n} \asconv  
	\mathbf{v}_1	
	~\quad~\mbox{and}~\quad~
	\frac{ \aegt(T_n) - n\mathbf{v}_1 }{\sqrt{n}} \wconv \mathcal{N}\left (\bzero,\Sigma \right ),
	\end{equation}
 where $\mathbf{v}_1= \frac{1}{16}( 2,2,1,3,1,7)$ and $\Sigma$ is given in Eq.~\eqref{eq:sigma:pda:urn}. In particular, as $n \to \infty$,
	\begin{equation}
	\frac{1}{n}(\pf'_n , \ch'_n) \asconv  \Big(\frac{1}{8}, \frac{1}{4} \Big) 
~~\quad~\mbox{and}~\quad~
\frac{ (\pf'_n , \ch'_n ) - n ( 1/8, 1/4)}{\sqrt{n}} \wconv \mathcal{N}\left (\bzero,~
	\frac{1}{64}  	 \begin{bmatrix*}[r]
	3 ~&~ 0 \\
	0 ~&~ 4
	\end{bmatrix*} 
	\right ).
	\end{equation}
\end{theorem}

\section{Proofs of Theorems 1 and 2}
\label{sec:tech:proofs}
In this section, we shall present the proofs of Theorems~\ref{thm1} and~\ref{thm2}. To this end, it is more natural to consider $Y_n:  = C_nU$, a linear transform of $C_n$. Next we introduce
\begin{equation}
\label{eq:def:mds}
\mds_n
= Y_n - \ee[Y_n \vert \FF_{n-1}].
\end{equation}
For $1 \le j \le d$, consider the following numbers
\begin{equation}
\label{eq:def:b}
b_{n, n}(j) =1~~\mbox{and} \quad b_{n, k}(j) = \prod_{\ell=k}^{n-1} (1+ \lambda_j/t_{\ell}) ~~~ \mbox{ for $0 \le k <n$.}
\end{equation}
Moreover,  we introduce the following diagonal matrix for $0\le k \le n$:
\begin{equation}
\label{eq:def:transf}
\trf_{n,k}= \text{diag}\left(b_{n, k}(1), \ldots , b_{n,k}(d)\right).
\end{equation}
Then we have the following key observation:
\begin{equation}
\label{eq:main:decomp}
Y_n   = Y_0  \trf_{n,0} +  \sum_{k=1}^n \mds_k\,  \trf_{n,k}.
\end{equation}

To see that~\eqref{eq:main:decomp} holds, let $Q_k = \mathbf{I} + t_{k-1}^{-1} R$ for $1 \le k \le n$, where $\mathbf{I}$ is the identity matrix. Then we have
$$  \ee[C_n \vert \FF_{n-1}] =  C_{n-1} + t_{n-1}^{-1} C_{n-1}  R  = C_{n-1} \left[\mathbf{I} +t_{n-1}^{-1} R \right] = C_{n-1} Q_n.$$
As $C_k - E[C_k \vert \FF_{k-1}]=\mds_k U^{-1}$ for $1\le k \le n$,   we have
\begin{align}
C_n  &=(C_n - \ee[C_n \vert \FF_{n-1}] ) + C_{n-1}Q_n
=\mds_n U^{-1} + C_{n-1}Q_n \nonumber \\
& =C_0 (Q_1\cdots Q_n)+ \mds_nU^{-1} +\sum_{k=1}^{n-1} \mds_k U^{-1}  (Q_{k+1}\cdots Q_n). \label{eq:C:decomp}
\end{align}
Since
\begin{align}
U^{-1} \Big(\prod_{\ell=k+1}^n Q_{\ell} \Big) U
= \prod_{\ell=k}^{n-1} \big(U^{-1} \left( I+ t^{-1}_{\ell}R \right) U \big)
= \prod_{\ell=k}^{n-1} \left( I+ t_{\ell}^{-1} \Lambda \right)= \trf_{n,k}
\end{align}
holds for $1\le k \le n$ and $Y_n=C_nU$,
it is straightforward to see that~\eqref{eq:main:decomp} follows from transforming \eqref{eq:C:decomp} by a right multiplication of $U$.

\bigskip
Next, we shall present several properties concerning $\mds_k$. To this end, consider the sequence of random vectors $\tau_k=\chi_k-\ee[\chi_k|\FF_{k-1}]$ for $k\ge 1$. Then $\{\tau_k\}_{k\ge 1}$ is a martingale difference sequence (MDS) in that $\ee[\tau_k|\FF_{k-1}]=\bzero$ almost surely.  
Hence 
$\ee[\tau_k]=\ee\big[\ee[\tau_k|\FF_{k-1}]\big]=\bzero$.
Furthermore, since the entries in $\chi_k$ is either $0$ or $1$ and  $\ee[\chi_k|\FF_{k-1}]=\widetilde{C}_{k-1}$, the random vector $\tau_k$ is also bounded. As a bounded martingale difference sequence, $\tau_k$ is uncorrelated. To see it, assuming that $\ell<k$, then we have
$$
\ee[\tau^\top_\ell \tau_k ]
=\ee\big[\ee[\tau^\top_\ell \tau_k |\FF_{k-1}] \big]
=\ee\big[\tau^\top_\ell\ee[ \tau_k |\FF_{k-1}] \big]
=\ee[\tau^\top_\ell \bzero ]
=\bzero^\top \bzero,
$$
where the first equality follows the total law of expectation and the second from $\tau_{\ell}$ is $\mathcal{F}_{k-1}$-measurable. A similar argument shows $\ee[\tau_\ell \tau^\top_k ]=0$. Consequently, we have the following expression showing that distinct $\tau_k$ and $\tau_\ell$ are uncorrelated:
\begin{equation}
\label{eq:tau:uncor}
\ee[\tau^\top_k \tau_\ell]=\bzero^\top \bzero~~\mbox{and}~~\ee[\tau_k \tau^\top_\ell]=0~~\quad \mbox{if $k\not = \ell$}.
\end{equation}
Moreover, putting
\[\Gamma_{k}: =  \text{diag}\big( \widetilde{C}_{k}\big) -    \widetilde{C}_{k}^\top  \widetilde{C}_{k},\]
then we have
\[\ee[\Gamma_{k}] =  \text{diag}\big( \ee[\widetilde{C}_{k}]\big) -   \ee \big[ \widetilde{C}_{k}^\top  \widetilde{C}_{k}\big].\]
Consequently, we have
\begin{eqnarray}
\ee[\tau^\top_k \tau_k | \FF_{k-1}]
&=&
\ee[
\big(\chi_k-\ee[\chi_k|\FF_{k-1}]\big)^\top
\big(\chi_k-\ee[\chi_k|\FF_{k-1}]\big)
| \FF_{k-1}]   \nonumber\\
&=& \ee[
\big(\chi^\top_k-\widetilde{C}_{k-1}^\top\big)
\big(\chi_k-\widetilde{C}_{k-1}\big)
| \FF_{k-1}]  \nonumber\\
&=&
\ee[\chi^\top_k\chi_k|\FF_{k-1}]
-\widetilde{C}_{k-1}^\top\ee[\chi_k|\FF_{k-1}]
-\ee[\chi_k^\top|\FF_{k-1}]\widetilde{C}_{k-1}
+\widetilde{C}_{k-1}^\top \widetilde{C}_{k-1} \nonumber \\
&=&
\ee[\chi^\top_k\chi_k|\FF_{k-1}]-
\widetilde{C}_{k-1}^\top \widetilde{C}_{k-1}
=\Gamma_{k-1},
\label{eq:tau:conditional:variation}
\end{eqnarray}
where the last equality follows from~\eqref{eq:chi:prop}. This implies
\begin{eqnarray}
\ee[\tau^\top_k \tau_k]
=\ee\big[\ee[\tau^\top_k \tau_k | \FF_{k-1}] \big]
=\ee[\Gamma_{k-1}].
\label{eq:tau:var}
\end{eqnarray}

Note that $\mds_k$ is a `linear transform' of $\tau_k$ in that combining
\eqref{eq:main:urn:dynamic} and \eqref{eq:def:mds} leads to
\begin{eqnarray}
\mds_k
&=&\big( C_k - \ee[C_k \vert \FF_{k-1}] \big)U
=\big( C_{k-1}+\chi_kR - \ee[C_{k-1}+\chi_kR \vert \FF_{k-1}] \big)U
\nonumber \\
&=&\big( \chi_{k} - \ee[\chi_k \vert \FF_{k-1}] \big)RU
=\tau_kRU =\tau_kU\Lambda.
\label{eq:tau:mds}
\end{eqnarray}
Note this implies that $\mds_k$ is a martingale difference sequence in that $\ee[\mds_k|\FF_{k-1}]=\bzero=\ee[\mds_k]$.
Furthermore,  by~\eqref{eq:tau:conditional:variation},  ~\eqref{eq:tau:var}, and~\eqref{eq:tau:mds} we have
\begin{equation}
\label{eq:ksi:conditiona:variation}
\ee \big[\mds_k^\top \mds_k|\FF_{k-1}\big]=\Lambda  U^\top  \Gamma_{k-1}  U\Lambda
~\quad~
\mbox{for $k\ge 1$.}
\end{equation}
Together with~\eqref{eq:tau:uncor},  for all $k,l\ge 1$ we have 
\begin{equation}
\ee [\mds_k^\top \mds_k]=\Lambda  U^\top  \ee[\Gamma_{k-1}]  U\Lambda, 
\quad \mbox{and} \quad
\ee [\mds_k^\top \mds_l]=\bzero^\top \bzero~ \mbox{if  $k\not =l$.}
\label{eq:M:main}
\end{equation}
Since $\mathbf{u}_1=U \mathbf{e}^\top_1 =\mathbf{e}^\top$ is a right eigenvector of $R$ corresponding to $s$, by~\eqref{eq:tau:mds} we have
\begin{align}
\mds_k\mathbf{e}^\top_1 =
\tau_kRU \mathbf{e}^\top_1
=\tau_kR\mathbf{u}_1
=s\tau_k\mathbf{u}_1
=s\tau_k\mathbf{e}^\top
=0~\mbox{for $k\ge 1$},
\label{eq:mds:first:col}
\end{align}
where the last equality follows from $\chi_k \mathbf{e}^\top =1$ and $\ee[\chi_k|\FF_{k-1}]\mathbf{e}^\top=\widetilde{C}_{k-1}\mathbf{e}^\top=1$.

\bigskip

Note that for $n>1$ and 
$\rho<1$, we have
\begin{equation}
\label{eq:Rieman:sum:bound}
\frac{1}{n}\sum_{k=1}^{n-1} \left( \frac{n}{k} \right)^\rho \le \frac{1}{1-\rho},
~~\mbox{and}~~
\lim_{n\to \infty} \frac{1}{n} \sum_{k=1}^n  \left( \frac{n}{k} \right)^{\rho}  =
\int_0^1 x^{-\rho} dx  = \frac{1}{1-\rho}.
\end{equation}
Furthermore, we present the following result on the entries of
$\trf_{n,k}$, whose proof is elementary calculus and included in the appendix.

\begin{lemma}
	\label{lem:entries:B}
	Under assumptions (A2) and (A3), there exists a constant $K$ such that  
	\begin{equation}
	\label{eq:b:upper:bound}
	|b_{n,0}(j)|\le Kn^{\lambda_j/s}
	~\quad~\mbox{and}~\quad~
	|b_{n,k}(j)| \le K(n/k)^{\lambda_j/s}
	\end{equation}
	hold	for $1\le j \le \udim$ and $1\le k \le n$.
	Furthermore,  we have	
	\begin{equation}
	\label{eq:b:second:moment}
	\lim_{n\to \infty} \frac{1}{n} \sum_{k=1}^n  b_{n,k}(i) b_{n,k}(j)
	= \frac{s}{s- \lambda_i -\lambda_j}
	~\quad\quad~
	\mbox{for $2\le i\le j \le \udim$.}
	\end{equation}	
\end{lemma}

With the last lemma, we have the following observation that will be key in the proof of Theorem~\ref{thm2}. 
\begin{corollary}
	\label{cor:conv:array:product}
	Assume that $\{Z_n\}$ is a sequence of random variables such that
	$$
	Z_n \pconv Z
	$$
	for a random variable $Z$. Then under assumptions (A2)-(A3),  for $2\le i \le j \le \udim$ we have
	\begin{equation}
	\label{eq:pconv:transform:cor}
	\frac{1}{n}\sum_{k=1}^n b_{n,k}(i)b_{n,k}(j)Z_k \pconv  \frac{s}{s- \lambda_i -\lambda_j} Z
	~\quad~
	\mbox{as $n\to \infty$}.
	\end{equation}
\end{corollary}

\begin{proof}
	Fix a pair of indexes $2\le i \le j \le \udim$. For simplicity, we put $a_{n,k}=b_{n,k}(i)b_{n,k}(j)$. Furthermore, let $\rho=(\lambda_i+\lambda_j)/s$, then $\rho<1$ and $1-\rho=(s-\lambda_i-\lambda_j)/s>0$. Then by Lemma~\ref{lem:entries:B}  we have 
	\begin{equation}
	\label{eq:conv:array:product:bound}
	\lim_{n\to \infty} \frac{1}{n}\sum_{k=1}^n a_{n,k}=\frac{1}{1-\rho},
	~\quad~
	\mbox{and}
	~\quad~
	|a_{n,k}| \le K\Big(\frac{n}{k}\Big)^\rho
	~\quad~
	~\mbox{for all $n\ge 1$ and $1\le k\le n$}.
	\end{equation}
	Furthermore, let $N_0$ be the smallest integer greater than 1 such that both
	$N_0> -(\lambda_i+t_0)/s$ and $N_0>-(\lambda_j+t_0)/s)$ hold. Then we have 	
	$a_{n,k}>0$ for all $n\ge k \ge N_0$. 
	
	We shall next show that
	\begin{equation}
	\label{eq:eq:pconv:transform:pf:1}
	\frac{1}{n}	\sum_{k=1}^n a_{n,k} \ee[|Z_k-Z|] \to 0.
	\end{equation}
	For simplicity, put $\beta_k:=\ee[|Z_k-Z|]$ for $k\ge 1$. Then $\{\beta_k\}_{k\ge 1}$  is a sequence of non-negative numbers which converges to $0$. Thus there exists a constant $K_1>0$ such that $\beta_k<K_1$ holds for all $k\ge 1$. Next, fix an arbitrary number $\epsilon>0$. By~\eqref{eq:conv:array:product:bound}, let $N_1=N_1(\epsilon)$ be the smallest integer greater than $N_0$ so that    so that $n^{-1}\sum_{k=1}^n a_{n,k}<\frac{1}{1-\rho}+\epsilon$ holds for all $n>N_1$. Since $1-\rho> 0$, the number 
	$\epsilon':=\frac{\epsilon(1-\rho)}{2(1+\epsilon(1-\rho))}$ is greater than $0$. Let $N_2$ be the smallest positive integer greater than $N_1$ so that $\beta_k<\epsilon'$ holds for all $k>N_2$. Now let $N$ be the smallest positive integer greater than $N_2$ so that
	$N\ge (2(K_1+\epsilon')KN_2/\epsilon)^{1/(1-\rho)}$ and $N\ge N_2(2(K_1+\epsilon')K/\epsilon)^{1/(1-\rho)}$ both hold. 
	Then for $n>N$ we have
	\begin{align*}
	\Big| \frac{1}{n}	\sum_{k=1}^n a_{n,k}\beta_k \Big|
	&\le \Big| \frac{1}{n}	\sum_{k=1}^{N_2} a_{n,k}\beta_k\Big|+\frac{1}{n}	\sum_{k=1+N_2}^n a_{n,k}\beta_k  
	\le \frac{K_1}{n}	\sum_{k=1}^{N_2} |a_{n,k}|+\frac{\epsilon'}{n}	\sum_{k=1+N_2}^n a_{n,k}\\
	& = \frac{K_1}{n}	\sum_{k=1}^{N_2} |a_{n,k}|-\frac{\epsilon'}{n}\sum_{k=1}^{N_2} a_{n,k} + \frac{\epsilon'}{n}	\sum_{k=1}^n a_{n,k}
	\\  & \le \frac{K_1+\epsilon'}{n}	\sum_{k=1}^{N_2} |a_{n,k}|+\frac{\epsilon'}{n}	\sum_{k=1}^n a_{n,k}\\
	&\le \frac{(K_1+\epsilon')KN_2\max(n^{\rho},(n/N_2)^{\rho})}{n}+\frac{\epsilon'}{n}	\sum_{k=1}^n a_{n,k} \\
	&\le \frac{\epsilon}{2}+\epsilon'\Big(\frac{1}{1-\rho}+\epsilon\Big) = \epsilon,
	\end{align*}
	from which~\eqref{eq:eq:pconv:transform:pf:1} follows. Here the first inequality follows from the triangle inequality and that $a_{n,k}\beta_k>0$ holds for $n\ge k >N_2\ge M$, the second inequality holds since $0\le \beta_k<K_1$ for $k\ge 1$ and $\beta_k<\epsilon'$ for $k>N_2$. Next, the third inequality holds since we have $\epsilon'(a_{n,k}+|a_{n,k}|)
	\ge 0$  for $1\le k \le n$. Furthermore,  the fourth inequality holds because by ~\eqref{eq:conv:array:product:bound} we have $|a_{n,k}|\le K\max(n^{\rho},(n/N_2)^{\rho})$ for $1\le k \le N_2$, and the last inequality follows from $2(K_1+\epsilon')KN_2n^\rho\le \epsilon n$ and $2(K_1+\epsilon')KN^{1-\rho}_2n^\rho\le \epsilon n$ in view of $n>N$.  	
	
	Finally, by~\eqref{eq:conv:array:product:bound} and~\eqref{eq:eq:pconv:transform:pf:1} it follows that
	\begin{equation}
	\frac{1}{n}\sum_{k=1}^n a_{n,k}Z \pconv \frac{1}{1-\rho}Z
	~\quad~\mbox{and}~\quad~
	\frac{1}{n}	\sum_{k=1}^n a_{n,k} (Z_k-Z) \pconv 0.
	\end{equation} 
	Therefore, we can conclude that 
	$$
	\frac{1}{n} \sum_{k=1}^n a_{n,k}Z_k=\frac{1}{n}\sum_{k=1}^n a_{n,k} Z+
	\frac{1}{n}\sum_{k=1}^n a_{n,k} (Z_k-Z)
	\pconv \frac{1}{1-\rho}Z,	
	$$
	as required.
	\epf
\end{proof}

\subsection{Proof of Theorem \ref{thm1}}

\begin{proof}
	
	
	Recall that $Y_n  = C_nU$ for $n\ge 1$. Hence, it is sufficient to show that
	\begin{equation}
	\label{eq:main:pf:lln}
	n^{-1}  Y_n  \asconv \pev \, \mathbf{e}_1
	\end{equation}
	because $\pev \: \mathbf{e}_1U^{-1}= \pev\: \mathbf{v}_1$ and $n^{-1}C_n=n^{-1}Y_nU^{-1}$. Furthermore, as the sequence of random vectors $n^{-1}C_n$ is bounded, its $L^r$ convergence follows from the almost sure convergence.
	
	To establish~\eqref{eq:main:pf:lln}, we restate the following decomposition from ~\eqref{eq:main:decomp} as below:
	\begin{equation}
	\label{eq:main:decomp:lln}
	Y_n   = Y_0  \trf_{n,0} +  \sum_{k=1}^n \mds_k\,  \trf_{n,k},
	\end{equation}
	where $\{\mds_k\}$ is the martingale difference sequence in~\eqref{eq:def:mds} and $\trf_{n,k}$ is the diagonal matrix in~\eqref{eq:def:transf}.
	
	Next we claim that
	\begin{equation}
	\label{eq:expect:asconv:lln}
	n^{-1}\ee[Y_n]\longrightarrow \pev\: \mathbf{e}_1~~\mbox{as $n\to \infty$}.
	\end{equation}
	Indeed, since $\ee[\mds_k]=\mathbf{0}$ implies
	$\ee[\mds_k\,  \trf_{n,k}]=\ee[\mds_k]\trf_{n,k}=\mathbf{0}$, by~\eqref{eq:main:decomp:lln} we have
	$\ee[Y_n]=Y_0  \trf_{n,0}.$
	Therefore the $j$-th entry in $\ee[Y_n]$, denoted by $y_{n,j}$,  is given by
	$$
	y_{n,j}=
	\ee[Y_n]\mathbf{e}_j^\top=Y_0  \trf_{n,0} \mathbf{e}_j^\top
	=b_{n,0}(j) Y_0  \mathbf{e}_j^\top
	~\mbox{for $1\le j \le \udim$}.
	$$
	When $j=1$, we have
	$$
	y_{n,1}=b_{n,0}(1) Y_0  \mathbf{e}_1^\top=(\mass_n/\mass_0)Y_0  \mathbf{e}_1^\top
	=(\mass_n/\mass_0)C_0 U \mathbf{e}_1^\top
	=(\mass_n/\mass_0)C_0 \mathbf{u}_1
	=(\mass_n/\mass_0)\mass_0=\mass_n,
	$$
	where we use the fact that $\mathbf{u}_1=\mathbf{e}^\top$ and hence
	$\mass_0=C_0 \mathbf{u}_1$. Therefore we have
	$y_{n,1}/n=\mass_n/n \to \pev$ as $n\to \infty$.
	On the other hand, for $2\le j \le \udim$, we have
	$$
	|y_{n,j}|=|b_{n,0}(j) Y_0  \mathbf{e}_j^\top|
	\le K_1|b_{n,0}(j)|
	\le K n ^{\lambda_j/\pev},
	$$
	where the last inequality follows from Lemma~\ref{lem:entries:B}. Since $\lambda_j<\pev$, it follows that $y_{n,j}/n \to 0$ as $n\to \infty$.
	This completes the proof of~\eqref{eq:expect:asconv:lln}.
	
	\bigskip
	For simplicity, let $Z_{n}:=Y_{n} - E[Y_{n}]$. Then we have $Y_n=Z_n+\ee[Y_n]$, by \eqref{eq:expect:asconv:lln} it follows that
	to establish \eqref{eq:main:pf:lln}, it remains to show that
	\begin{equation}
	\label{eq:Z:asconv}
	Z_n/n \asconv \mathbf{0},
	\end{equation}
	Denote the $j$-th entry in $Z_n$ by $Z_{n,j}$, then from~\eqref{eq:main:decomp:lln}   we have
	\begin{equation}
	\label{eq:Z:componet}
	Z_{n,j}
	=\sum_{k=1}^n  (\mds_k  \trf_{n,k} )\mathbf{e}^\top_j
	=  \sum_{k=1}^n b_{n,k}(j) \mds_k \mathbf{e}^\top_j.
	\end{equation}
	Since~\eqref{eq:Z:asconv} is equivalent to
	\begin{equation}
	\label{eq:Z:component:asconv}
	\frac{Z_{n,j}}{n} \asconv 0
	~~~~~~~~\mbox{~~~~for $1\le j \le \udim$},
	\end{equation}
	the remainder of the proof is devoted to establishing~\eqref{eq:Z:component:asconv}.
	\bigskip
	
	It is straightforward to see that~\eqref{eq:Z:component:asconv} holds for $j=1$ because by~\eqref{eq:mds:first:col} and~\eqref{eq:Z:componet} we have
	$$
	Z_{n,1}=\sum_{k=1}^n b_{n,k}(j) \mds_k \mathbf{e}^\top_1
	=0.
	$$
	Thus in the remainder of the proof, we may assume that $2\le j \le \udim$ holds.
	Note that
	\begin{eqnarray*}
		\ee\big[ Z_{n,j}^2\big]
		&=&  \ee \bigg [\bigg(\sum_{k=1}^n  b_{n,k}(j) \mds_k \mathbf{e}^\top_j    \bigg)^2 \bigg]
		=
		\ee \big [ \sum_{k,l=1}^n  b_{n,k}(j) b_{n,l}(j) \mathbf{e}_j \mds^\top_k  \mds_{l} \mathbf{e}^\top_j   \big] \\
		&=& \ee \big [ \sum_{k=1}^n  b^2_{n,k}(j) \mathbf{e}_j \mds_k^\top  \mds_{k} \mathbf{e}^\top_j  \big]
		=  \sum_{k=1}^n  b^2_{n,k}(j) \ee \big [\mathbf{e}_j \mds_k^\top  \mds_{k} \mathbf{e}^\top_j  \big].
	\end{eqnarray*}
	Here the third equality follows form~\eqref{eq:M:main}.  As $\ee[\mathbf{e}_j \mds^\top_k  \mds_{k} \mathbf{e}^\top_j]$,  the $(j,j)$-entry of matrix $\ee[\mds^\top_k  \mds_{k}]$,  is  bounded above by a constant $K_1$ in view of~\eqref{eq:M:main}, there exists constants $K_2$ and $K$ so that 
	\begin{align*}
	\ee\big[ Z_{n,j}^2\big]
	&\le K_1  \sum_{k=1}^n |b_{n,i}(j)|^2
	\leq  K_2 \ \sum_{k=1}^n  \left(\frac{n}{k} \right)^{2  \lambda_{j}/s}
	=K_2+  K_2n \ \sum_{k=1}^{n-1}  \frac{1}{n}\left(\frac{k}{n} \right)^{-2  \lambda_{j}/s} 
	\\&
	\le K_2+ \frac{K_2n}{1-2 \lambda_{j}/s} 
	\le Kn
	\end{align*}
	holds for all $n\ge 1$. Here the second inequality follows from Lemma~\ref{lem:entries:B} 
	and the third one from~\eqref{eq:Rieman:sum:bound} in view of $\lambda_j < s/2$ for  $2 \le j \le d$.  
	
	Since $\ee(Z_{n,j})=0$,  for $\epsilon >0$ using the Chebychev inequality we get
	\begin{equation}
	\label{eq:chebychev}
	\pp\left ( \left |Z_{n,j} \right | >n \epsilon \right) \leq \frac{K}{n \epsilon^2}
	~\quad\mbox{for all $n\ge 1$}.
	\end{equation}
	Consider the subsequence $Z'_{n,j}$ of $Z_{n,j}$ with $Z'_{n,j}=Z_{n^2, j}$ for $ n \ge 1$.   Then for $\epsilon >0$ we have
	\begin{align*}
	\sum_{n=1}^\infty \pp\left ( \frac{|Z'_{n,j}|}{n^2} > \epsilon \right) =
	\sum_{n=1}^\infty \pp\left ( \left |Z_{n^2,j}\right | > n^2\epsilon \right)
	\leq \sum_{n=1}^\infty  \frac{K  } {n^2 \epsilon^2} <\infty,
	\end{align*}
	where the first inequality follows from~\eqref{eq:chebychev}.
	Thus,  by the Borel-Cantelli Lemma, it follows that
	\begin{equation}
	\label{eq:sub:asconv}
	{n^{-2}}{Z'_{n,j}} \asconv 0.
	\end{equation}
	
	
	Next, consider
	\begin{align*}
	\Delta_{n,j}  &:=
	\max_{n^2 \le k < (n+1)^2 } | Z_{k,j} -Z'_{n,j} |  =
	\max_{n^2 \le k < (n+1)^2 } | Z_{k,j} -Z_{n^2,j} |
	= \max_{1 \leq k \leq 2n } |Z_{n^2+k,j} -Z_{n^2,j} |.
	\end{align*}
	Since for each $\ell> 0$, elements of $\chi_\ell$ and $RU$ are all bounded above, there exists a constant $K$ independent of $\ell$ and $j$ so that
	\begin{align*}
	|Z_{\ell+1,j}-Z_{\ell,j}|
	&=
	|\big((Y_{\ell+1}-\ee[Y_{\ell+1}])-(Y_{\ell}-\ee[Y_{\ell}])\big)U \mathbf{e}^\top_j| \\
	&=|\big(Y_{\ell+1}-Y_{\ell}\big)-\big(\ee[Y_{\ell+1}-Y_{\ell}]\big)U
	\mathbf{e}^\top_j| 
	= |\big(\chi_{\ell+1}-\ee[\chi_{\ell+1}] \big) RU\mathbf{e}^\top_j| \le K.
	\end{align*}
	Consequently, we have
	\begin{align*}
	\Delta_{n,j}& =\max_{0 \leq k \leq 2n } |Z_{n^2+k,j} -Z_{n^2,j} |
	\le \max_{1 \leq k \leq 2n } \sum_{\ell=1}^k |Z_{n^2+\ell,j} -Z_{n^2+\ell-1,j} |
	\le \max_{1 \leq k \leq 2n } \sum_{\ell=1}^k K =2nK,
	\end{align*}
	and hence
	\begin{equation}
	\label{eq:diff:asconv}
	n^{-2}{\Delta_{n,j}} \asconv 0.
	\end{equation}
	Now, for each $k>0$, considering the natural number
	$n$ with  $n^2 \le k < (n+1)^2$, then we have
	\begin{equation}
	\label{eq:sub:bound}
	\frac{\left|Z_{k,j}\right|}{k} \le  \frac{\left| Z_{k,j}   -Z_{n^2,j}\right| }{k} +\frac{\left|Z_{n^2,j}\right|}{k}
	\leq  \frac{\Delta_{n,j}}{n^2}+\frac{\left|Z_{n^2,j}\right|}{n^2}=
	\frac{\Delta_{n,j}}{n^2}+\frac{\left|Z'_{n,j}\right|}{n^2}.
	\end{equation}
	Note that when $k\to \infty$, the natural number $n$ satisfying  $n^2 \le k < (n+1)^2$ also approaches to $\infty$.  Thus combining~\eqref{eq:sub:asconv},~\eqref{eq:diff:asconv}, and~\eqref{eq:sub:bound} leads to
	\begin{equation}
	\label{eq:bc}
	k^{-1}{Z_{k,j}} \asconv 0  \qquad \text{when $k\to \infty$},
	\end{equation}
	which completes the proof of~\eqref{eq:Z:component:asconv}, and hence also the theorem.
	\epf
\end{proof}
\vspace{1em}

\subsection{Proof of Theorem \ref{thm2}}

\begin{proof}
	For each $n\ge 1$, consider the following two sequences of random vectors:
	$$
	X_{n,k}: = n^{-1/2} \mds_k   \trf_{n,k} \quad
	~\mbox{and} \quad
	S_{n,k}: =  \sum_{\ell=1}^k X_{n,\ell}
	~~\quad~~\mbox{for $1\le k \le n$,}
	$$
	where $\{\mds_k\}_{k\ge 1}$ is the martingale difference sequence in~\eqref{eq:def:mds} and $\trf_{n,k}$ is the diagonal matrix in~\eqref{eq:def:transf}.
	Then for each $n\ge 1$, the sequence $\{X_{n,k}\}_{1\le k \le n}$ is a  martingale difference sequence, and  $\{S_{n,k}\}_{1\leq k \leq n}  $ is a  mean zero martingale.
	Recalling that $Y_n  = C_nU$, then by~\eqref{eq:main:decomp} we have
	\begin{equation}
	\label{eq:S:Y}
	S_{n,n}= n^{-1/2}\sum_{k=1}^n \mds_k   \trf_{n,k} =
	n^{-1/2}\big(  Y_n - \ee[Y_n ]  \big).
	\end{equation}
	
	Consider the normal distribution $\mathcal{N}({\mathbf 0},   \widetilde{\Sigma})$ is with mean vector $\mathbf{0}$ and variance-covariance matrix
	\begin{equation}
	\widetilde{\Sigma}: = \ \sum_{i, j=2}^d \frac{s\lambda_i \lambda _j  \mathbf{u}_i^\top \text{diag}(\mathbf{v}_1) \mathbf{u}_j  }{s-\lambda_i -\lambda_j}  \mathbf{e}_i^\top \mathbf{e}_j.
	\end{equation}
	One key step in our proof is to show that
	\begin{equation}
	\label{eq:S:normal}
	S_{n,n} \wconv   \mathcal{N}({\mathbf 0},   \widetilde{\Sigma}).
	\end{equation}
	
	Before establishing~\eqref{eq:S:normal}, we shall first show that the theorem follows from it. To this end, we claim that
	\begin{equation}
	\label{eq:Z:0}
	Z_n:= n^{-1/2} \left( \ee[Y_n]-ns \mathbf{e}_1 \right) \asconv \mathbf{0} ~~\mbox{ with  $n\to \infty$.}
	\end{equation}
	Indeed, we have $Z_{n}\mathbf{e}^\top_1=n^{-1/2}(t_n- ns)=n^{-1/2}t_0 \to 0$.
	Furthermore,  by Lemma~\ref{lem:entries:B}  there exists a constant $K$ such that
	\begin{eqnarray*}
		|Z_n\mathbf{e}^\top_j| = n^{-1/2}|Y_{0,j} b_{n,0}(j)|=n^{-1/2}Y_{0,j} |b_{n,0}(j)|
		\leq n^{-1/2}Y_{0,j}  K n ^{\lambda_j/s}
		~~\quad~~\mbox{ for $2\le j \le d$.}
	\end{eqnarray*}
	As $\lambda_j/s <1/2$, it follows that $|Z_{n}\mathbf{e}^\top_j| \to 0$ for all $1\le j \le d$, and hence~\eqref{eq:Z:0} holds.
	Consequently, we have
	\begin{equation}
	\label{eq:Y:normality}
	n^{-1/2} \left( Y_n - ns \mathbf{e}_1\right)= n^{-1/2} \left( Y_n - E[Y_n] \right)  + Z_n
	=S_{n,n}+Z_n
	\wconv  N({\mathbf 0},   \widetilde{\Sigma}).
	\end{equation}
	Here the second equality follows from~\eqref{eq:S:Y}; convergence in distribution follows from the Slutsky theorem (see,e.g., ..to add) in view of~\eqref{eq:S:normal} and \eqref{eq:Z:0}.
	Since  $n^{-1/2} ( C_n - ns \mathbf{v}_1)
	=n^{-1/2} \left( Y_n - ns \mathbf{e}_1\right) V$ with $V=U^{-1}$,
	by~\eqref{eq:Y:normality} and the fact that a linear transform of a normal vector is also normal (a citation, todo) we have
	\begin{equation}
	n^{-1/2} ( C_n - ns \mathbf{v}_1) \wconv N({\mathbf 0},   \Sigma),
	\end{equation}
	where
	\begin{equation}
	\Sigma = V^\top \widetilde \Sigma \, V=V^\top \Big( \sum_{i, j=2}^d \frac{s\lambda_i \lambda _j  \mathbf{u}_i^\top \text{diag}(\mathbf{v}_1) \mathbf{u}_j  }{s-\lambda_i -\lambda_j}  \mathbf{e}_i^\top \mathbf{e}_j \Big) V
	=\sum_{i, j=2}^d \frac{s\lambda_i \lambda _j  \mathbf{u}_i^\top \text{diag}(\mathbf{v}_1) \mathbf{u}_j  }{s-\lambda_i -\lambda_j}  \mathbf{v}_i^\top \mathbf{v}_j,
	\end{equation}
	which shows indeed that the theorem follows from~\eqref{eq:S:normal}.

	\bigskip
	In the remainder of the proof we shall establish~\eqref{eq:S:normal}. To this end, considering
	$$
	\Phi(n):=\sum_{k=1}^n
	\ee\big[X_{n,k}^\top X_{n,k} |\FF_{k-1}\big]
	=\frac{1}{n}   \sum_{k=1}^n    \trf_{n,k} \ee [  \mds_k ^\top \mds_k |\FF_{k-1} ]
	\trf_{n,k},
	$$
	and we shall next show that 
	\begin{equation}
	\label{eq:pconv:xcv:phi}
	\Phi(n) \pconv \widetilde{\Sigma}.
	\end{equation}
	Let $\Gamma = \text{diag}(\mathbf{v}_1) -  \mathbf{v}_1^\top \mathbf{v}_1$.
	Note that for  $ 2 \le i,j\le d$, we have $\mathbf{v}_1 \mathbf{u}_i = 0=\mathbf{v}_1 \mathbf{u}_j $ in view of~\eqref{eq:u:v:innerproduct}, and 
	hence
	\begin{align*}
	\frac{s\lambda_i \lambda_j \mathbf{u}_i^\top \Gamma \mathbf{u}_j }{s- \lambda_i-\lambda_j}
	=  \frac{s\lambda_i \lambda_j \mathbf{u}_i^\top  (\text{diag}(\mathbf{v}_1) - \mathbf{v}_1^\top\mathbf{v}_1 )  \mathbf{u}_j }{s- \lambda_i-\lambda_j}
	=\frac{s\lambda_i \lambda_j \mathbf{u}_i^\top \text{diag}(\mathbf{v}_1)  \mathbf{u}_j }{s- \lambda_i-\lambda_j}.
	\end{align*}
	Therefore~\eqref{eq:pconv:xcv:phi} is equivalent to 
	\begin{equation}
	\mathbf{e}_i \Phi(n) \mathbf{e}_j^\top \pconv \begin{cases}
	\dfrac{s\lambda_i \lambda _j \ \mathbf{u}_i^\top \Gamma \mathbf{u}_j}{s- \lambda_i-\lambda_j}
	\  , &~ 2 \le i, j \le \udim, \\
	0, & ~ \text{if  $i=1$ or $j=1$},
	\end{cases}
	\label{eq:pconv:xcv:point}
	\end{equation}
	
	Since  $\trf_{n,k}$ is a diagonal matrix and $\mathbf{e}_1\mds_k^\top=0$ in view of~\eqref{eq:mds:first:col}, this implies
	$$
	\mathbf{e}_1 \Phi(n)
	=\frac{1}{n}   \sum_{k=1}^n   \mathbf{e}_1 \trf_{n,k} \ee [  \mds_k ^\top \mds_k \,|\, \FF_{k-1}]   \trf_{n,k}
	=\frac{1}{n}   \sum_{k=1}^n    b_{n,k}(1) \ee [  \mathbf{e}_1\mds_k ^\top \mds_k \,|\, \FF_{k-1}]   \trf_{n,k}
	=\bzero.
	$$
	A similar argument shows $\Phi_{n}\mathbf{e}^\top_1=\bzero$, and hence~\eqref{eq:pconv:xcv:point} holds for $i=1$ or $j=1$.
	It remains to consider the case $2\le i,j\le \udim$.
	Since
	$$
	\widetilde{C}_k \mconv \mathbf{v}_1
	~~\quad \mbox{and}~~\quad
	\widetilde{C}^\top_k\widetilde{C}_k \mconv
	\mathbf{v}^\top_1\mathbf{v}_1
	$$
	hold in view of Theorem \ref{thm1}, by~\eqref{eq:ksi:conditiona:variation} we have
	\[
	\ee [  \mds_k ^\top \mds_k |\FF_{k-1} ]  =
	\Lambda U^\top  \Gamma_{k-1} U\Lambda
	\mconv
	\Lambda U^\top  \Gamma  U\Lambda
	\]
	and hence
	\begin{equation}
	\label{eq:gamma:mconv:pf}
	\lambda_i \lambda_j \mathbf{u}_i^\top \Gamma_{k} \mathbf{u}_j
	\pconv \lambda_i \lambda_j \mathbf{u}_i^\top \Gamma \mathbf{u}_j
	~~\mbox{as $k \to \infty$.}
	\end{equation}

	As both $ \trf_{n, k} $ and $\Lambda$ are diagonal matrices,
	we have
	\begin{align}
	\frac{1}{n} & \sum_{i=k}^n \mathbf{e}_i
	\trf_{n,k}  (\Lambda U^\top  \Gamma_{k-1} U\Lambda)   \trf_{n,k}
	\ \mathbf{e}^\top_j
	= \frac{1}{n} \sum_{k=1}^n   b_{n,k}(i) b_{n,k}(j) \mathbf{e}_i\Lambda U^\top  \Gamma_{k-1} U\Lambda\mathbf{e}^\top_j    \nonumber \\
	& \quad = \frac{\lambda_i \lambda_j}{n}
	\sum_{k=1}^n   b_{n,k}(i) b_{n,k}(j)  \mathbf{u}_i^\top \Gamma_{k-1} \mathbf{u}_j
	\pconv
	\frac{s\lambda_i \lambda_j \mathbf{u}_i^\top \Gamma \mathbf{u}_j }{s- \lambda_i-\lambda_j} ,
	\label{eq:pf:clt:var:pconv}
	\end{align}
	where the convergence follows from Corollary~\ref{cor:conv:array:product} and ~\eqref{eq:gamma:mconv:pf}.

	\bigskip
	
	Since $S_{n,n}$ is a mean $\bzero$ random vector and $\trf_{n,k}$ is a diagonal matrix, we have
	\begin{align*}
	\var\left[S_{n,n} \right]
	&=\ee[S^\top_{n,n}S_{n,n}]
	=  \frac{1}{n}   \sum_{k,\ell=1}^n    \trf^\top_{n,k} \ee [  \mds_k ^\top \mds_\ell ]   \trf_{n,\ell}
	=\frac{1}{n}   \sum_{k=1}^n    \trf_{n,k} \ee [  \mds_k ^\top \mds_k ]
	\trf_{n,k} \\
	&= \sum_{k=1}^n \ee[X^\top_{n,k} X_{n,k}]
	=\ee[\Phi(n)]
	\end{align*}
	where the third equality follows from~\eqref{eq:M:main}.
	Furthermore, an argument similar to the proof of~\eqref{eq:pconv:xcv:phi} shows that
	\begin{equation}
	\lim_{n\to \infty} \var(S_{n,n}) = \widetilde{\Sigma}.
	\nonumber
	\end{equation}
	Therefore $\widetilde{\Sigma}$ is positive semi-definite because the matrix $\var(S_{n,n})$ is necessarily positive semi-definite for each $n\ge 1$.

	Following the Cram\'er-Wold device for multivariate central limit theorem
	(see, e.g.~\citet[Theorem 3.10.6]{durrett2019probability}), fix an arbitrary row vector  $\cwt=(w_1,\cdots,w_{\udim})$ in $\mathbb{R}^{\udim}\setminus \{\bzero\}$ and put $s_{n,k}=S_{n,k} \cwt^\top$ and $x_{n,k}=X_{n,k}\cwt^\top$. Furthermore,
	since the matrix $\widetilde{\Sigma}$ is positive semi-definite, we can introduce $\sigma^2:=\cwt\, \widetilde{\Sigma}\, \cwt^\top \ge 0$. Then for establishing~\eqref{eq:S:normal} it suffices to show that
	\begin{equation}
	\label{eq:local:conv:cw}
	s_{n,n} \wconv N(0, \sigma^2).
	\end{equation}
	Since $\{x_{n,k}\}_{1\le k \le n}$ is a martingale difference sequence and  $\{s_{n,k}\}_{1\leq k \leq n}  $ is an array of mean zero martingale,  the martingale central limit theorem (see, e.g.~\citet[ Corollary 3.2]{hall2014martingale}) implies that~\eqref{eq:local:conv:cw} follows from
	\begin{equation}
	\label{eq:local:variance:mclt}
	\gamma_n:= \sum_{k=1}^{n} \ee\left [ \left | x_{n,k} \right | ^2
	\vert \FF_{k-1} \right] \pconv \sigma^2
	~~\quad~\mbox{as $n\to \infty$}
	\end{equation}
	and   the conditional Lindeberg-type  condition holds, that is, for every  $\epsilon >0$
	\begin{equation}
	\label{eq:local:Lindeberg:mclt}
	\gamma^*_n:=\sum_{k=1}^n \ee\left [ \left | x_{n,k} \right | ^2
	\mathbb{I}_{A_{n,k,\epsilon}} \vert \FF_{k-1} \right] \pconv 0
	~~\quad~\mbox{as $n\to \infty$}
	\end{equation}
	where $\mathbb{I}_{A_{n,k,\epsilon}}$ is the indicator variable on $A_{n,k,\epsilon}: = \{ |x_{n,k}| > \epsilon\}$.

	Now~\eqref{eq:local:variance:mclt} follows from
	\begin{align}
	\gamma_n & =  \sum_{k=1}^{n} \ee\left [\cwt X_{n,k}^\top X_{n,k}\cwt^\top |\FF_{k-1}
	\right]
	= \cwt \sum_{k=1}^{n} \ee\left [ X_{n,k}^\top X_{n,k} |\FF_{k-1}
	\right] \cwt^\top \nonumber \\
	&= \cwt \Phi_n  \cwt^\top
	\pconv \cwt\, \widetilde{\Sigma}\, \cwt^\top=
	\sigma^2,
	\label{eq:var:conv:gamma}
	\end{align}
	where the convergence follows from~\eqref{eq:pconv:xcv:phi}.

	\bigskip

	To see that~\eqref{eq:local:Lindeberg:mclt} holds, by~\eqref{eq:tau:mds} we have 
	$$
	X_{n,k}=\sum_{j=1}^{\udim} X_{n,k}\mathbf{e}^\top_j \mathbf{e}_j=\sum_{j=1}^{\udim} n^{-1/2} \lambda_j\,b_{n, k}(j) \tau_k \bu_j \mathbf{e}_j,
	~~\quad~\mbox{$1\le k \le n$}.
	$$
	In particular, we have $X_{n,k}(1)=0$ because $\tau_k \bu_1=0$ holds for $k\ge 1$ in view of~\eqref{eq:mds:first:col}. Consequently, we have
	\begin{equation}
	\label{eq:x:def}
	x_{n,k}=X_{n,k} \cwt^\top=\sum_{j=2}^{\udim} n^{-1/2} w_j \lambda_j b_{n, k}(j) \tau_k \bu_j.
	\end{equation}
	Putting $\rho=\lambda_2/s$,  then $\lambda_j/s\le \rho<1/2$ holds for $2\le j \le \udim$ in view of (A2) and (A4). Furthermore, there exists a constant $K_0>0$ independent of $n$ and $k$ such that
	\begin{equation}
	\label{eq:x:abs}
	|x_{n,k}| \leq 
	\sum_{j=2}^{\udim} n^{-1/2} |w_j \lambda_j \tau_k \bu_j| |b_{n, k}(j)|
	\leq K_0 n^{-1/2}(n/k)^{\rho}
	\leq K_0n^{-1/2}\max(1, n^{\rho})
	\end{equation}
	holds for $1\le k \le n$. Here the second inequality follows from Lemma~\ref{lem:entries:B}  and the fact that   $| w_j \lambda_j \tau_k \mathbf{u}_j|$ is  bounded above by a constant independent of $k$. 
	The last inequality follows from the fact that
	$(n/k)^{\rho} \le \max \big( (n/1)^{\rho},(n/n)^{\rho} \big)
	$.
	Now let $A'_{n,\epsilon}: = \{ 
	K_0n^{-1/2}\max(1, n^{\rho})> \epsilon\}$, which it is either $\emptyset$ if $n$ is sufficient large or the whole probability space otherwise. Then by~\eqref{eq:x:abs} we have $A_{n,k,\epsilon}\subseteq A'_{n,\epsilon}$ and hence for all $\epsilon>0$ and each $n$, we have $\mathbb{I}_{A_{n,k,\epsilon}}\le \mathbb{I}_{A'_{n,\epsilon}}$ for all $1\le k \le n$. Furthermore, since $\rho<1/2$ and $K_0>0$, we have
	\begin{equation}
	\label{eq:conv:events:epsilon}
	\ee[\mathbb{I}_{A'_{n,\epsilon}}]=\pp(A'_{n,\epsilon}) \to 0~~\mbox{as $n \to \infty$}.
	\end{equation}
	Consequently, we have
	\begin{align}
	\ee[\gamma^*_n] &= \ee\Big[ \sum_{k=1}^{n}
	\ee\big[ \left | x_{n,k} \right | ^2  \mathbb{I}_{A_{n,k,\epsilon}}
	\vert \FF_{k-1} \big]  \Big]
	\leq  \ee\Big[ \sum_{k=1}^{n} \ee\big[ \left | x_{n,k} \right | ^2 \mathbb{I}_{A'_{n,\epsilon}} \vert \FF_{k-1} \big]  \Big] \\
	&=  \ee\Big[ \Big( \sum_{k=1}^{n} \ee\big[ \left | x_{n,k} \right | ^2  \vert \FF_{k-1} \big] \Big) \mathbb{I}_{A'_{n,\epsilon}}  \Big]
	=\ee\left[  \gamma_n\mathbb{I}_{A'_{n,\epsilon}} \right]
	\\
	&=  \ee\big[  \gamma_n \big]  \ee\big[ \mathbb{I}_{A'_{n,\epsilon}}\big]
	\to 
	0,~~\mbox{as $n\to \infty$}
	\end{align}
	where we have used the fact that $\mathbb{I}_{A'_{n,\epsilon}}$ is $\FF_{n}$-measurable and independent of $\FF_{n}$ (and all its sub-sigma-algebras); the convergence follows from~\eqref{eq:var:conv:gamma}~and~\eqref{eq:conv:events:epsilon}. Since $\gamma^*_n$ is almost surely non-negative, this completes the proof of~\eqref{eq:local:Lindeberg:mclt}, the last step in the proof of the  theorem.
	\epf
\end{proof}

\section{Discussion}
\label{sec:disc}


Inspired by a martingale approach developed in~\citet{BaiHu2005},  we present 
in this paper the strong law of large numbers and the central limit theorem for a family of the P\'olya urn models in which negative off-diagonal entries are allowed in their replacement matrices. 
This leads to a unified approach to proving corresponding limit theorems for the joint vector of cherry and pitchfork counts under the YHK and the PDA models, namely, the joint random variable converges almost surely to a deterministic vector and converges in distribution to a bivariate normal distribution. 
Interestingly, such convergence results also hold for unrooted tees and do not depend on the initial trees used in the generating process.

The results presented here also lead to several broad directions that may be interesting to explore in future work. The first direction  concerns a more detailed analysis on convergence. For instance, the central limit theorems present here should be extendable to a functional central limit theorem, a follow-up project that we will pursue. Furthermore, it remains to establish the rate of convergence for the limit theorems. For example, a law of the iterated logarithm would add considerable information to the strong law of large numbers by providing a more precise estimate of the size of the almost sure fluctuations of the random sequences in Theorems~\ref{thm:yhk:color} and ~\ref{thm:pda:color}. 




The second direction concerns whether the results obtained here can  be extended to other tree statistics and tree models. For example,  the two tree models considered here, the YHK and the PDA, can be regarded as special cases of some more general tree generating models, such as Ford’s alpha model (see, e.g.~\citet{chen2009new}) and the Aldous beta-splitting model (see, e.g.~\citet{aldous96a}). Therefore, it is of interest to extend our studies on subtree indices to these two models as well. Furthermore, instead of cherry and pitchfork statistics, we can consider more general subtree indices such as $k$-pronged nodes and $k$-caterpillars~\citep{rosenberg06a,chang2010limit}.


Finally, it would be interesting to study tree shape statistics for several recently proposed  graphical structures in evolutionary biology. For instances, one can consider aspects of tree shapes that are related to the distribution of branch lengths \citep{ferretti2017decomposing,arbisser2018joint} or relatively ranked tree shapes~\citep{kim2020distance}. 
Furthermore, less is known about shape statistics in phylogenetic networks, in which non-tree-like signals such as lateral gene transfer and viral recombinations are accommodated~\citep{bouvel2020counting}. Further understanding of their statistical properties could  help us design more complex evolutionary models that may in some cases provide a better framework for understanding real datasets. 





\begin{acknowledgements}
K.P. Choi acknowledges the support of Singapore Ministry of Education Academic Research Fund R-155-000-188-114. 
The work of Gursharn Kaur was supported by NUS Research Grant R-155-000-198-114.
We thank the Institute for Mathematical Sciences, National University of Singapore where this project started during the discussions in the \emph{Symposium in Memory of Charles Stein}.

\end{acknowledgements}

%
%

\section*{Appendix}
In the appendix we present a proof of Lemma~\ref{lem:entries:B}
concerning bounds on the entries of
$\trf_{n,k}$. 
To this end, we start with the following observation.

\begin{lemma}
	\label{lem1}
	For $\lambda \in \mathbb{R}$, $\ell \in \mathbb{R}_{> 0}$, and two non-negative integers $m$ and $n$ with $n\ge m$, put 
	$$
	F_m^m(\ell, \lambda)=1,
	~\quad~
	\mbox{and}
	~\quad~
	F_m^n(\ell, \lambda):= \prod_{i=m}^{n-1} \left( 1+\frac{\lambda}{\ell + i} \right)
	~\quad~
	\mbox{for $n>m$.}
	$$   Then we have
	\begin{equation}
	\label{eq:lem1:limit}
	\lim_{m \to \infty} 
	\sup_{n\ge m}
	\left(\frac{m}{n} \right)^{\lambda} F_m^n(\ell, \lambda) =1.
	\end{equation}
	Furthermore,  there exists a positive constant $K=K(\lambda,\ell)$  such that
	\begin{equation}
	\label{eq:lem1:bound}
	\left|F_m^n(\ell, \lambda)  \right|   \le
	K \left(n/m\right)^\lambda
	~~\mbox{for all $1\le m\le n$.}
	\end{equation}
\end{lemma}
\begin{proof}
	
	Since  the lemma holds for $\lambda=0$ in view of $F_m^n(\ell,0)=1$, we will assume that $\lambda\not = 0 $ in the remainder of the proof.  
	For simplicity, put $L:=\max\big(1,-(\ell+\lambda)\big)$.

	First we shall establish~\eqref{eq:lem1:limit}. To this end, we may assume $m>L$, and hence $m+\ell+\lambda>0$. 
	Furthermore,  recall the following result on the ratio of gamma functions
	(see, e.g.~\citep[P.398]{jameson2013inequalities}
	  for a proof for the case $y>0$, which can be easily extended to the other case $y\le 0$): for a fixed number $y\in \mathbb{R}$, we have 
	\begin{equation}
	\label{eq:gamma:ratio:limit}
	\lim_{x\to \infty}  \frac{\Gamma(x+y)}{x^y\Gamma(x)}=1.
	\end{equation}
	Therefore, putting
	\begin{equation}
	G_{m,k}:=\frac{\Gamma(m+k+\ell+\lambda)}{(m+k)^\lambda
		\Gamma(m+k+\ell)}
	~\quad~\mbox{for integer $k\ge 0$},
	\nonumber
	\end{equation}
	then	we have
	\begin{equation}
	\label{eq:gamma:sup:ratio:limit:1}
	\lim_{m\to \infty}  \ln \big(G_{m,0}\big) =0,
	~\quad~\mbox{and hence}~\quad~
	\lim_{m\to \infty} \sup_{k\ge 0} \ln \big(G_{m+k,0}\big)=0.
	\end{equation}
	Here the second limit holds because the limit of $\ln \big(G_{m,0}\big)$ being $0$ implies that its limit superior is also $0$. 
	Together with $G_{m,k}=G_{m+k,0}$ for  $k\ge 0$, this leads to 
	\begin{equation}
	\label{eq:gamma:sup:ratio:limit:2}
	\lim_{m\to \infty} \sup_{k\ge 0} \ln \big(G_{m,k}\big) 
	=	\lim_{m\to \infty} \sup_{k\ge 0} \ln \big(G_{m+k,0}\big) 
	=0.
	\end{equation}
	Since 
	\begin{align}
	\left(\frac{m}{m+k} \right)^{\lambda} F_m^{m+k}(\ell, \lambda)
	&=
	\left(\frac{m}{m+k} \right)^{\lambda}  \frac{ \Gamma(m+k+\ell+\lambda)\Gamma(m+\ell)}{ \Gamma(m+k+\ell)\Gamma(m+\ell+\lambda)} 
	=\frac{G_{m,k}}{G_{m,0}}
	\end{align}
	holds for each integer $k\ge 0$, we have
	\begin{eqnarray*}
		\lim_{m \to \infty} \sup_{n\ge m} \ln \Big(\left(\frac{m}{n} \right)^{\lambda} F_m^n(\ell, \lambda) \Big)
		&=&
		\lim_{m \to \infty} \sup_{k\ge 0} \ln \Big(\Big(\frac{m}{m+k} \Big)^{\lambda} F_m^{m+k}(\ell, \lambda) \Big)\ \\
		&=&\lim_{m \to \infty} \sup_{k\ge 0} \Big(\ln(G_{m,k})-\ln(G_{m,0})\Big) \\
		&=& \lim_{m \to \infty} \sup_{k\ge 0} \ln(G_{m,k})- \lim_{m \to \infty} \ln(G_{m,0})\\
		&=&0,
	\end{eqnarray*}
	where the last equality follows from ~\eqref{eq:gamma:sup:ratio:limit:1} and
	~\eqref{eq:gamma:sup:ratio:limit:2}. This completes the proof of~\eqref{eq:lem1:limit}.
	
	\bigskip
	
	Next, we shall establish~\eqref{eq:lem1:bound}. To this end we 
	assume $m<n$, $m+\ell+\lambda\not =0$, and $n-1+\ell+\lambda\not =0$  as otherwise it clearly holds. Now consider  the following three cases:
	
	\medskip
	\noindent
	Case 1:  $1\le m \le n-1 < L$, and hence $n-1+\ell+\lambda < 0$. Let $A=\{(\alpha,\beta)\,| \alpha, \beta \in \mathbb{N}; 1\le \alpha \le \beta \le1 -\ell-\lambda\}$ 
	be the  finite subset of $\mathbb{N}\times \mathbb{N}$ whose size depends on $\ell$ and $\lambda$, and consider the constant
	$$K_1:=\max_{(\alpha,\beta)\in A} \big\{|F_\alpha^\beta(\ell,\lambda)|
	\left(
	\alpha/\beta
	\right)^\lambda\big\}.
	$$ Since $(m,n)\in A$, it follows that
	$|F_m^n(\ell,\lambda)|\le K_1(n/m)^\lambda$ holds.
	
	\noindent
	Case 2:  $m\ge L$ and hence $m+\ell+\lambda >0$. Note that in this case we have $F_m^n(\ell,\lambda)>0$. Furthermore, an argument similar to the proof of ~\eqref{eq:lem1:limit} shows that for each $m\ge L$ we have
	$$
	\lim_{n \to \infty} 
	\left(\frac{m}{n} \right)^{\lambda} F_m^n(\ell, \lambda) =\frac{1}{G_{m,0}},
	$$
	and hence there exists a constant $K'_m$ depending on $m, \ell,\lambda$ so that 
	$F_m^n(\ell,\lambda)\le K'_m (n/m)^\lambda$
	holds. Furthermore, by~\eqref{eq:lem1:limit} it follows that there exists a constant $M=M(\ell,\lambda)$ and a constant $K_0=K_0(\ell,\lambda)$ so that 
	$F_m^n(\ell,\lambda)\le K_0 (n/m)^\lambda$	holds for all $m> M$. Therefore, for  the constant
	$$
	K_2:=\max\{K_0,K'_1,\cdots,K'_{M} \},
	$$
	which depends only on $\ell$ and $\lambda$, we have 
	$F_m^n(\ell,\lambda)\le K_2 (n/m)^\lambda$
	for all $L\le m \le n$.

	\medskip
	\noindent
	Case 3:  $1\le m<L<n-1$ and hence $m+\ell+\lambda <0<n-1+\ell+\lambda$.
	Note this implies $L>1$ and we may further assume that $L$ is not an integer as otherwise $F_m^n(\ell,\lambda)=0$ follows.  Let $p$ be the (necessarily positive) largest integer less than $L$. Then $1\le m\le p<L$ and by  Case 1 we have $|F_m^p(\ell,\lambda)|\le K_1(p/m)^\lambda$. Furthermore, as
	$p+1>L$, by Case 2 we have $|F_{p+1}^n(\ell,\lambda)|\le K_2(n/p+1)^\lambda$.
	Therefore, considering the constant $K_3=\max\{K_1K_2,2^{-\lambda} K_1K_2\}$, which depends on only $\ell$ and $\lambda$,  we have
	\begin{eqnarray*}
		|F_m^n(\ell,\lambda)| &=& |F_m^p(\ell,\lambda) F_{p+1}^n(\ell,\lambda)| 
		\le  K_1K_2 \left(\frac{p}{m}\right)^\lambda \left(\frac{n}{p+1}\right)^\lambda
		\\&=& 
		K_1K_2 \left(\frac{p}{p+1}\right)^\lambda \left(\frac{n}{m}\right)^\lambda 
		\le K_3 \left(\frac{n}{m}\right)^\lambda.
	\end{eqnarray*}
	The last inequality follows since  
	$(\frac{p}{p+1})^\lambda\le 1$ holds for $\lambda>0$, and 
	$(\frac{p}{p+1})^\lambda\le 2^{-\lambda}$ for  $\lambda<0$. 
	\epf
\end{proof}

\vspace{1em}

With Lemma~\ref{lem1}, we now present a proof of Lemma~\ref{lem:entries:B}. 


\noindent
{\bf Proof of~Lemma~\ref{lem:entries:B}.}
	Recall that by (A3) we have $t_{\ell}=t_0+\ell s$ for $\ell\ge 1$, 
	and hence 
	$$
	b_{n, k}(j) = \prod_{\ell=k}^{n-1} \Big(1+ \frac{\lambda_j}{t_0+\ell s} \Big)
	=\prod_{\ell=k}^{n-1} \Big( 1+ \frac{\lambda_j/s}{(t_0/s)+\ell} \Big)
	=F_k^n\Big( \frac{t_0}{s},\frac{\lambda_j}{s} \Big)
	$$
	holds for $1\le j \le \udim$ and $1\le k < n$.
	Noting that $b_{n,n}(j)=1$,  by Lemma~\ref{lem1} there exists a constant $K'_j$ such that 
	$|b_{n,k}(j)|\le K'_j(n/k)^{\lambda_j/s}$ holds for $1\le k \le n$. Now let
	$a_j=1+(\lambda_j/t_0)$ and put $K_j=\max(K'_j,K'_j|a_j|)$. Then we have 	$|b_{n,0}(j)|\le Kn^{\lambda_j/s}$ in view of $b_{n,0}(j)=a_jb_{n,1}(j)$. This establishes ~\eqref{eq:b:upper:bound} by choosing $K=\max(K_1,\cdots,K_\udim)$.

	Next, we shall show~\eqref{eq:b:second:moment}. To this end, fix a pair of indices $2\le i \le j \le \udim$, and put $\rho_i=\lambda_i/s$ and $\rho_j=\lambda_j/s$. 
	Then by (A2) we have $\rho:=\rho_i+\rho_j<1$ and $1-\rho=(s-\lambda_i-\lambda_j)/s$. Furthermore, consider 
	$$
	\Delta_{n}:=\frac{1}{n}\sum_{k=1}^n \left( \frac{n}{k} \right)^{\rho_j}  
	\Big(
	b_{n,k}(i)-\left( \frac{n}{k} \right)^{\rho_i}
	\Big)
	~~
	\mbox{and}
	~~
	\Delta^*_{n}:=\frac{1}{n}\sum_{k=1}^n b_{n,k}(i) 
	\Big(
	b_{n,k}(j)-\left( \frac{n}{k} \right)^{\rho_j}
	\Big).
	$$	
	Then we have 
	$$
	\Delta_{n}+\Delta^*_{n}= \frac{1}{n}\sum_{k=1}^n 
	\Big( 
	b_{n,k}(i)b_{n,k}(j)-\left( \frac{n}{k} \right)^{\rho}
	\Big). 
	$$
	By~\eqref{eq:Rieman:sum:bound} it suffices to show that both $\Delta_{n}\to 0$ and $\Delta^*_{n} \to 0$ as $n\to \infty$.
	
	By~\eqref{eq:b:upper:bound}, we have $|b_{n,k}(i)| \le K(n/k)^{\rho_i}$ and $|b_{n,k}(j)| \le K(n/k)^{\rho_j}$. We shall first show that $\Delta_{n} \to 0$ as $n\to \infty$. To this end, let
	$$
	H_{n,k}:=\left( \frac{n}{k} \right)^{\rho_j}  
	\Big(
	b_{n,k}(i)-\left( \frac{n}{k} \right)^{\rho_i}
	\Big)
	=\left( \frac{n}{k} \right)^{\rho}  
	\Big(
	\Big( \frac{k}{n} \Big)^{\rho_i} b_{n,k}(i)-1
	\Big).
	$$
	Then we have $|H_{n,k}|\le (K+1) (n/k)^{\rho}$ for all $1\le k \le n$.
	Consider $\epsilon>0$. Then it follows that $\epsilon':=\frac{\epsilon(1-\rho)}{2(2-\rho)}>0$.  By~\eqref{eq:lem1:limit} in Lemma~\ref{lem1}, we have
	\begin{equation}
	\lim_{k \to \infty} 
	\sup_{n\ge k}
	\left(\frac{k}{n} \right)^{\rho_i} b_{n,k}(i) =
	\lim_{k \to \infty} 
	\sup_{n\ge k}
	\left(\frac{k}{n} \right)^{\rho_i} F_k^n\Big(\frac{t_0}{s}, \rho_i\Big) =1.
	\nonumber
	\end{equation}
	Therefore, there exists a constant $M$ such that 
	$|H_{n,k}|\le \epsilon' (n/k)^{\rho}$ holds for all $n\ge k \ge M$. 
	Moreover, let $N$ be the smallest integer greater than $M$ so that 
	$N>[2(K+1)M/\epsilon]^{1/(1-\rho)}$ and $N>M[2(K+1)/\epsilon]^{1/(1-\rho)}$ both hold. 
	Then for $n>N$ we have 
	\begin{align*}
	\frac{1}{n}\sum_{k=1}^n |H_{n,k}| 
	&
	\le \frac{\epsilon'}{n}\sum_{k=M+1}^n \left( \frac{n}{k} \right)^{\rho}
	+\frac{1}{n}\sum_{k=1}^{M} |H_{n,k}| 
	\le 
	\frac{\epsilon'}{n}\sum_{k=1}^n \left( \frac{n}{k} \right)^{\rho}
	+\frac{1}{n}\sum_{k=1}^{M} |H_{n,k}| 
	\\&
	\le \frac{\epsilon'(2-\rho)}{1-\rho}+\frac{K+1}{n^{1-\rho}}\sum_{k=1}^{M} \left( \frac{1}{k} \right)^{\rho}
	\le \frac{\epsilon}{2}+\frac{\epsilon}{2}
	= \epsilon,
	\end{align*}
	where in the third inequality we use the fact that~\eqref{eq:Rieman:sum:bound} implies  $$
	\frac{1}{n}\sum_{k=1}^n \left( \frac{n}{k} \right)^{\rho}\le \frac{1}{n}+\frac{1}{1-\rho}\le 1+\frac{1}{1-\rho}=\frac{2-\rho}{1-\rho}.
	$$
	Therefore it follows that $\Delta_{n} \to 0 $ as $n \to \infty$. 
	Since $|b_{n,k}(i)| \le K(n/k)^{\rho_i}$, a similar argument can be adopted to show that $\Delta^*_{n} \to 0 $ as $n \to \infty$,   completing the proof of Lemma~\ref{lem:entries:B}.
	\epf

\bibliographystyle{spbasic}      
\bibliography{0_StatPhyloTree.bib}

\end{document}